\documentclass{amsart}
\usepackage{amsmath,amssymb,amscd,amsfonts}
\usepackage[all,cmtip]{xy}
\usepackage{graphicx}

\newtheorem{theorem}{Theorem}[section]

\newtheorem{thm}{Theorem}
\newtheorem{prop}[thm]{Proposition}
\newtheorem{lem}[thm]{Lemma}
\newtheorem{cor}[thm]{Corollary}

\theoremstyle{definition}

\theoremstyle{remark}
\newtheorem{remark}[theorem]{Remark}

\numberwithin{equation}{section}

\newcommand{\A}{\mathbb{A}}
\newcommand{\bK}{\mathbb{K}}
\newcommand{\Q}{\mathbb{Q}}
\newcommand{\R}{\mathbb{R}}
\newcommand{\Z}{\mathbb{Z}}

\newcommand{\cA}{\mathcal{A}}
\newcommand{\cB}{\mathcal{B}}
\newcommand{\cI}{\mathcal{I}}
\newcommand{\cO}{\mathcal{O}}
\newcommand{\cP}{\mathcal{P}}
\newcommand{\cT}{\mathcal{T}}
\newcommand{\cBT}{\mathcal{BT}}

\newcommand{\BM}{\mathrm{BM}}
\newcommand{\cl}{\mathrm{cl}}
\newcommand{\diag}{\mathrm{diag}}
\newcommand{\Ext}{\mathrm{Ext}}
\newcommand{\fin}{\mathrm{fin}}
\newcommand{\GL}{\mathrm{GL}}
\newcommand{\Hom}{\mathrm{Hom}}
\newcommand{\id}{\mathrm{id}}
\newcommand{\Ker}{\mathrm{Ker}}
\newcommand{\Lat}{\mathrm{Lat}}
\newcommand{\Latbar}{\overline{\mathrm{Lat}}}
\newcommand{\Map}{\mathrm{Map}}
\newcommand{\PGL}{\mathrm{PGL}}
\newcommand{\sgn}{\mathrm{sgn}}
\newcommand{\Tor}{\mathrm{Tor}}

\newcommand{\bsl}{\backslash}
\newcommand{\eps}{\epsilon}
\newcommand{\pmone}{{\{\pm 1\}}}

\newcommand{\inj}{\hookrightarrow}
\newcommand{\surj}{\twoheadrightarrow}
\newcommand{\resp}{resp.\ }
\newcommand{\xto}[1]{\xrightarrow{#1}}
\newcommand{\wt}[1]{\widetilde{#1}}
\newcommand{\wh}[1]{\widehat{#1}}

\newcommand{\BT}{|\cB\cT_\bullet|}
\newcommand{\BTbar}{\overline{|\cB\cT_\bullet|}}
\newcommand{\XbartimesEGamma}{\overline{|\cB\cT_\bullet|}\times |E\Gamma_\bullet|}
\newcommand{\GBT}{\Gamma \backslash \cB\cT_\bullet}
\newcommand{\GgBT}{\Gamma\backslash|\cB\cT_\bullet|}
\newcommand{\GgBTEG}{\GgBT\times |E\Gamma_\bullet|}
\newcommand{\GgbarBTEG}{\Gamma\backslash\cgBTdot \times |E\Gamma_\bullet|}
\newcommand{\gBTdot}{|\mathcal{BT}_\bullet|}
\newcommand{\cgBTdot}{\overline{|\mathcal{BT}_\bullet|}}

\newcommand{\C}{\mathbb{C}}
\newcommand{\Ind}{\mathrm{Ind}}
\newcommand{\St}{\mathrm{St}}
\newcommand{\cC}{\mathcal{C}}
\newcommand{\cH}{\mathcal{H}}
\newcommand{\Spec}{\mathrm{Spec}\,}
\newcommand{\cF}{\mathcal{F}}
\newcommand{\rank}{\mathrm{rank}}
\newcommand{\cL}{\mathcal{L}}
\newcommand{\Supp}{\mathrm{Supp}\,}
\newcommand{\cD}{\mathcal{D}}
\newcommand{\Mat}{\mathrm{Mat}}
\newcommand{\divi}{\mathrm{div}}
\newcommand{\Flag}{\mathrm{Flag}}
\newcommand{\Vertex}{\mathrm{Vert}}

\begin{document}

\title[modular symbols in positive characteristic]{The Borel-Moore homology of an arithmetic quotient of 
the Bruhat-Tits building of PGL of a 
non-archimedean local field in positive characteristic
and modular symbols}
\author{Satoshi Kondo, Seidai Yasuda}
\maketitle

\begin{abstract}We study the homology and the Borel-Moore homology
with coefficients in $\Q$ of a quotient (called arithmetic quotient)
of the Bruhat-Tits building of $\PGL$ of a nonarchimedean local field
of positive characteristic by
an arithmetic subgroup (a special case of the general definition
in Harder's article (Invent.\ Math.\ 42, 135-175 (1977)).

We define an analogue of modular symbols in this context
and show that the image of 
the canonical map from homology to Borel-Moore homology
is contained in the
sub $\Q$-vector space generated by the modular symbols.

By definition, the limit of the Borel-Moore homology as the arithmetic group
becomes small
is isomorphic to the space of $\Q$-valued automorphic forms 
that satisfy certain conditions at a distinguished (fixed) place
(namely that it is fixed by 
the Iwahori subgroup and the center at the place).
We show that the limit of the homology with $\C$-coefficients
is identified 
with the subspace consisting of cusp forms.
We also describe an irreducible subquotient of the 
limit of Borel-Moore homology as an induced representation
in a precise manner and give a multiplicity one type result.
\end{abstract}
\section{Introduction}

Let us state our result slightly more precisely than in the abstract.
Let us give the setup.
We let $F$ denote a global field of positive characteristic.
Let $C$ be a proper smooth curve over a finite field
whose function field is $F$.
Let $\infty$ be a place of $F$ and let $K=F_\infty$ denote
the local field at $\infty$.

Let $\cBT_\bullet$ be the Bruhat-Tits building for 
$\PGL_d$ of $K$ for a positive integer $d$.
It is a simplicial complex of dimension $d-1$.
Let $\Gamma \subset \GL_d(K)$ be 
an arithmetic subgroup (see the definition
in Section~\ref{sec:4.1.1}).
We consider the homology,
the Borel-Moore homology,
and the canonical map from homology to Borel-Moore homology
of the quotient $\Gamma \backslash \cBT_\bullet$.

A building is made of (in particular, a union of)
subsimplicial complexes called apartments.
These are labeled (not one-to-one) by  
the set of bases of $K^{\oplus d}$.
In the Borel-Moore homology of an apartment,
there
is defined its fundamental class.
When the class corresponds to an $F$-basis, 
using the pushforward map for Borel-Moore 
homology, we obtain a class in the Borel-Moore homology
of the quotient $\GBT$.    
We regard this class 
as an analogue of a modular symbol.
The first of our main results in its rough form may be 
stated as follows.  See Theorem~\ref{lem:apartment} for
the precise form.
\begin{thm}
\label{thm:intro}
The image of the canonical injective homomorphism 
\[
H_{d-1}(\Gamma\backslash \cBT_\bullet, \Q)
\to 
H_{d-1}^\BM(\Gamma\backslash \cBT_\bullet, \Q)
\]
is contained in the subspace generated by the 
classes of the apartments, which correspond to
$F$-bases.
\end{thm}

We refer the reader to the nice introduction in \cite{AR}
for a short exposition on the classical modular symbols.
In the paper \cite{AR}, they consider (the analogue of) the
modular symbols for the quotient by some congruence 
subgroup of $\mathrm{SL}_d(\R)$.  Our result may be 
regarded as a non-archimedean analogue of a part of 
their result, especially Proposition 3.2, p.246.  

Let us remark on the difference from the 
archimedean case and the difficulty in our case.
In \cite{AR}, 
they use Borel-Serre bordification
(compactification), on which the group $\Gamma$
acts freely.  In the non-archimedean case, 
there exist several compactifications
of (the geometric realization of) 
the Bruhat-Tits building of $\PGL$
(the reader is referred to the introduction in \cite{We2}),
but the action is not free.
We therefore work with the equivariant (co)homology groups.

There is a reason to use Werner's compactification 
and not other compactifications.   The key 
in the proof of our result is the explicit 
construction of 
the map (4) of 
Section~\ref{sec:4 and 5} 
at the level of chain complexes in the indicated 
direction.  
We do not have an analogous construction
for other compactifications, 
since for the construction we use  
the continuous map $s(v_1,\dots, v_d)$
(see Lemma~\ref{lem:cont s}) 
which uses the interpretation of the 
geometric points of the compactification
of the Bruhat-Tits building
as the set of semi-norms.
A further advantage is that the map
$s(v_1,\dots, v_d)\times [g_0,\dots, g_{d-1}]$
readily defines a class in the equivariant
homology.

The modular symbols considered here have an 
application in our other 
paper \cite{KY:Zeta elements}.  
The setup and 
the statements in Sections~\ref{sec:2}
and~\ref{sec:BT} of this article
already appeared there.
We reproduce them here for convenience. 
That paper uses the results 
Lemma~\ref{lem:arithmetic} and 
Theorem~\ref{lem:apartment}.

We turn to our second result.
We let $A=H^0(C \setminus \{\infty\}, \cO_C)$.
Here we identified a closed point of $C$ and
a place of $F$.
We write $\wh{A}=\varprojlim_I A/I$, where
the limit is taken over the nonzero ideals of $A$.
We let  $\A^\infty=\wh{A} \otimes_A F$ denote the ring of 
finite adeles.
For an open compact subgroup
$\bK \subset \GL_d(\A^\infty)$,
let $X_{\bK, \bullet}= \GL_d(F) \backslash (\GL_d(\A^\infty)/\bK \times \cBT_\bullet)$
(see Section~\ref{subsec:X_K} for the precise definition).
It is the finite disjoint union of spaces of the form $\Gamma \backslash \cBT_\bullet$
for some arithmetic subgroup $\Gamma$.
Then we have the following result.  See Proposition~\ref{prop:66_3} and 
Theorem~\ref{7_prop1} for the relevant notation.
\begin{thm}
\label{thm:intro2}
\begin{enumerate}
\item 
We have 
\[
\varinjlim_\bK 
H_{d-1}(X_{\bK, \bullet}, \C)
\cong
\bigoplus_\pi \pi^\infty\]
as representations of $\GL_d(\A)$
where $\pi=\pi^\infty \otimes \pi_\infty$
runs over the irreducible cuspidal automorphic 
representations of 
$\GL_d(\A)$ such that $\pi_\infty$ is isomorphic to the Steinberg representation of $\GL_d(K)$.

\item
Let $\pi = \pi^\infty \otimes \pi_\infty$ be an irreducible
smooth representation of $\GL_d(\A^\infty)$
such that $\pi^\infty$ appears as a subquotient of
$\varinjlim_\bK H_{d-1}^\mathrm{BM}(X_{\bK,\bullet},\C)$.
Then there exist an integer $r \ge 1$, 
a partition $d=d_1 + \cdots + d_r$ of $d$,
and irreducible cuspidal automorphic representations $\pi_i$
of $\GL_{d_i}(\A)$ for $i=1,\ldots,d$ which satisfy
the following properties:
\begin{enumerate}
\item For each $i$ with $0 \le i \le r$,
the component $\pi_{i,\infty}$ at $\infty$ of $\pi_i$ is 
isomorphic to the
Steinberg representation of $\GL_{d_i}(F_\infty)$.
\item Let us write $\pi_i = \pi_i^\infty \otimes \pi_{i,\infty}$.
Let $P \subset \GL_d$ denote the standard parabolic subgroup
corresponding to the partition $d=d_1 + \cdots + d_r$.
Then $\pi^\infty$ is isomorphic to a subquotient of the unnormalized 
parabolic induction $\Ind_{P(\A^\infty)}^{\GL_d(\A^\infty)}
\pi_1^\infty \otimes \cdots \otimes \pi_r^\infty$.
\end{enumerate}
Moreover for any subquotient $H$ of $\varinjlim_\bK H_{d-1}^\mathrm{BM}(X_{\bK,\bullet},\C)$
which is of finite length as a representation of $\GL_d(\A^\infty)$,
the multiplicity of $\pi$ in $H$ is at most one.
\end{enumerate}
\end{thm}

Theorem~\ref{thm:intro2} (1) follows from a result of Harder \cite{Harder}
and the classification theorem due to Moeglin and Waldspurger \cite{MW}
(see Remark~\ref{rmk:HMW} for a sketch).
We give another proof which does not rely on Harder's result.
Theorem~\ref{thm:intro2} (2) without Condition (a) follows from a 
general theorem of Langlands.
In a forthcoming paper by the second author,
an application of this result 
(using Condition (a) and the multiplicity one)
will be given.  
The multiplicity result seems new in that, although we have a restriction that $\pi_\infty$ is isomorphic to the Steinberg representation, a subquotient of the Borel-Moore homology may not be contained in the discrete part of $L^2$-automorphic forms.
For the proof of Theorem~\ref{thm:intro2}(2),
we study the quotient building using the interpretation as the 
moduli of vector bundles with certain structures on $C$.
The tools that appear are the same as in \cite{Gra}
but we actually study the quotient space while in \cite{Gra}
they consider only some orbit spaces since they ask only for
finite generation of cohomology groups.   

One technical problem, 
which was already addressed by Prasad
in the paper by Harder \cite[p.140, Bemerkung]{Harder}, 
is that the 
quotient of the Bruhat-Tits building may not be a simplicial 
complex.  In Section~\ref{sec:2} we give 
a generalization of the notion of simplicial complexes
so as to include those quotients.

We warn that our use of the term Borel-Moore homology
is not a common one.  We give some justification 
in Section~\ref{sec:cellular}.

The paper is organized as follows.
In Section~\ref{sec:simplicial complex}, 
we consider simplicial complexes.  Actually,
we generalize the definition of a simplicial 
complex in the usual sense (which we call 
strict simplicial complex).  One reason for
doing so is that while the Bruhat-Tits building 
itself is canonically a (strict) simplicial complex,
the quotient is not so canonically.   
We also redefine (co)homology in an orientation
free manner.  This is because  
the Bruhat-Tits building is not 
canonically oriented.
In Section~\ref{sec:BT}, 
we recall the definitions of the Bruhat-Tits 
building and the apartments. 
Besides the recollections, we give a 
construction of the fundamental class of
an apartment in the Borel-Moore homology
group.  This serves as an analogue of a cycle
(from 0 to $i\infty$ in the upper half plane,
for example) in the classical case.
In Section~\ref{sec:71}, we give the definition
and some properties of an arithmetic group.
The first of our main results 
(Theorem~\ref{lem:apartment}) 
is stated in this section.
Section~\ref{Modular Symbols} is devoted  
to the proof of Theorem~\ref{lem:apartment}.
The key to the proof is the construction 
of the maps (4) and (5) in Section~\ref{sec:4 and 5}.

Sections~\ref{section8} and~\ref{sec:BMquot}
are devoted to Theorem~\ref{thm:intro2},
and are independent of Section~\ref{Modular Symbols}.
We give the definition of the simplicial complex
$X_{\bK, \bullet}$ and make precise the relation between the
limit of the (Borel-Moore) homology of $X_{\bK,\bullet}$
as $\bK$ becomes small and the space 
of ($\Q$-valued) automorphic forms.
The main result of Section~\ref{section8}
is Proposition~\ref{7_prop2}.
In Section~\ref{sec:BMquot}, we prove Theorem~\ref{thm:intro2}(2),
or Theorem~\ref{7_prop1}.
The contents of Sections~\ref{sec:locfree}
and~\ref{sec:HN}
are reformulation of \cite{Gra}.
The aim of Section~\ref{sec:props} is to 
state Propositions~\ref{7_prop1b} and~\ref{7_prop2}.
The proofs are given in Section~\ref{sec:pf1b} and in
Section~\ref{sec:pf2} respectively.
The proof of Theorem~\ref{7_prop1} using
Propositions~\ref{7_prop1b} and~\ref{7_prop2}
is given in Section~\ref{sec:pfthm}.

\section{Simplicial complexes and their (co)homology}
\label{sec:2}
The material of this section (except for the remark in Section~\ref{sec:cellular})
appeared in Sections 3 and 5 of \cite{KY:Zeta elements}.
We collected them for the convenience of readers.

\subsection{Simplicial complexes}
\label{sec:simplicial complex}
\subsubsection{}
Let us recall the notion of (abstract) simplicial complex.
A simplicial complex is a pair $(Y_0,\Delta)$ of a set $Y_0$
and a set $\Delta$ of finite subsets of $Y_0$ which satisfies
the following conditions:
\begin{itemize}
\item If $S \in \Delta$ and $T\subset S$, then $T \in \Delta$.
\item If $v \in Y_0$, then $\{v \} \in \Delta$.
\end{itemize}
In this paper we call a simplicial complex in the sense above
a strict simplicial complex, and use the terminology 
``simplicial complex" in a little broader sense,
since we will treat as simplicial complexes
some arithmetic quotients of Bruhat-Tits building, in which
two different simplices may have the same set of vertices.
Bruhat-Tits building itself is a strict simplicial complex.
Our primary example of a (nonstrict) simplicial complex
is $\Gamma \backslash \cBT_\bullet$ for 
an arithmetic group $\Gamma$ 
(to be defined in Section~\ref{sec:def Gamma}).
%
%

We adopt the following definition of a simplicial complex:
a simplicial complex is a collection 
$Y_\bullet = (Y_i)_{i \ge 0}$ of the sets 
indexed by non-negative integers, equipped with
the following additional data
\begin{itemize}
\item a subset $V(\sigma) \subset Y_0$ 
with cardinality $i+1$, for each $i \ge 0$ and
for each $\sigma \in Y_i$
(we call $V(\sigma)$ the set of vertices of $\sigma$),
and
\item an element in $Y_j$, for each $i \ge j \ge 0$, 
for each $\sigma \in Y_i$,
and for each subset $V' \subset V(\sigma)$ with cardinality
$j+1$ (we denote this element in $Y_j$ by the symbol
$\sigma \times_{V(\sigma)} V'$ and call it the face of
$\sigma$ corresponding to $V'$)

\end{itemize}
which satisfy the following conditions:
\begin{itemize}
\item For each $\sigma \in Y_0$, the equality
$V(\sigma) = \{\sigma\}$ holds,
\item For each $i \ge 0$, for each $\sigma \in Y_i$,
and for each non-empty subset $V' \subset V(\sigma)$,
the equality $V(\sigma \times_{V(\sigma)} V') = V'$ holds.
\item For each $i \ge 0$ and for each $\sigma \in Y_i$,
the equality $\sigma \times_{V(\sigma)} V(\sigma) = \sigma$
holds, and
\item For each $i \ge 0$, for each $\sigma \in Y_i$,
and for each non-empty subsets $V', V'' \subset V(\sigma)$
with $V'' \subset V'$, the equality 
$(\sigma \times_{V(\sigma)} V')\times_{V'} V'' = 
\sigma \times_{V(\sigma)} V''$ holds.
\end{itemize}
Let us call the element of the form 
$\sigma\times_{V(\sigma)} V'$ for $j$ and $V'$
as above, 
the $j$-dimensional face of $\sigma$ corresponding to $V'$.
We remark here that the symbol $\times_{V(\sigma)}$
does not mean a fiber product in any way.

Any strict simplicial complex gives a simplicial 
complex in the sense above in the following way.
Let $(Y_0,\Delta)$ be a strict simplicial complex.
We identify $Y_0$ with the set of subsets of $Y_0$
with cardinality $1$.
For $i \ge 1$ let $Y_i$ denote the set of the elements
in $\Delta$ which has cardinality $i+1$ as a subset of $Y_0$.
For $i \ge 1$ and for $\sigma \in Y_i$, 
we set $V(\sigma)= \sigma$ regarded
as a subset of $Y_0$. 
%
For a non-empty subset $V \subset V(\sigma)$,
of cardinality $i'+1$, we set 
$\sigma \times_{V(\sigma)} V = V$ regarded
as an element of $Y_{i'}$. 
Then it is easily checked that 
the collection $Y_\bullet = (Y_i)_{i \ge 0}$ 
together with the assignments $\sigma \mapsto V(\sigma)$
and $(\sigma, V) \mapsto \sigma \times_{V(\sigma)} V$
forms a simplicial complex.

Let $Y_\bullet$ and $Z_\bullet$ be simplicial complexes.
We define a map from $Y_\bullet$ to $Z_\bullet$ to
be a collection 
$f=(f_i)_{i \ge 0}$ of maps $f_i : Y_i \to Z_i$ of sets 
which satisfies the following
conditions:
\begin{itemize}
\item for any
$i \ge 0 $ and for any $\sigma \in Y_i$, 
the restriction of $f_0$ to $V(\sigma)$ is injective
and the image of $f|_{V(\sigma)}$ is equal to the set
$V(f_i(\sigma))$, and 
\item for any
$i \ge j \ge 0$, for any $\sigma \in Y_i$, and for any
non-empty subset $V' \subset V(\sigma)$ with cardinality $j+1$
we have $f_j(\sigma \times_{V(\sigma)} V')
= f_i(\sigma) \times_{V(f_i(\sigma))} f_0(V')$.
\end{itemize}
\subsubsection{}
There is an alternative, less complicated, 
equivalent definition of a simplicial complex 
in the sense above, which we will describe in this
paragraph. As it will not be used in this article, 
the reader may skip this paragraph.
For a set $S$, let $\cP^\fin(S)$ denote the
category whose object are the non-empty finite subsets of
$S$ and whose morphisms are the inclusions.
Then giving a simplicial complex in our sense
is equivalent to giving a pair $(Y_0,F)$ of a set $Y_0$
and a presheaf $F$ of sets on $\cP^\fin(Y_0)$
such that $F(\{\sigma \}) = \{\sigma\}$ holds for every
$\sigma \in Y_0$. This equivalence is explicitly
described as follows: given a simplicial complex $Y_\bullet$, 
the corresponding $F$ is the presheaf 
which associates, to a non-empty finite subset $V \subset Y_0$
with cardinality $i+1$, the set of elements $\sigma \in Y_i$
satisfying $V(\sigma)=V$.

This alternative definition of a simplicial complex is smarter,
nevertheless we have adopted the former definition
since it is nearer to the definition 
of a simplicial complex in the usual sense.

\subsection{Homology and cohomology}
\label{sec:BM homology}

Usually the homology groups of $Y_\bullet$ are defined 
to be the homology groups of a complex $C_\bullet$ 
whose component in degree $i$ is the free abelian group 
generated by the $i$-simplices of $Y_\bullet$. 
For a precise definition of the boundary homomorphism
of the complex $C_\bullet$, 
we need to choose an orientation of each simplex. 
In this paper we adopt an alternative, 
equivalent definition of homology groups which 
does not require any choice of orientations. 
The latter definition seems a little complicated at first glance, 
however it will soon turn out to be a better way 
for describing the (co)homology of the arithmetic quotients 
Bruhat-Tits building, which seems to have no canonical, 
good choice of orientations.

\subsubsection{Orientation}
\label{sec:orientation}
We recall in 
Sections~\ref{sec:orientation} and~\ref{sec:def homology} 
the precise definitions of the (co)homology,
the cohomology with compact support and 
the Borel-Moore homology of a simplicial complex.
When computing (co)homology, one usually fixes an orientation
of each simplex once and for all, but we do not.
This results in an apparently different definition,
but they indeed agree with the usual definition.

We introduce the notion of orientation of a simplex.
Let $Y_\bullet$ be a simplicial complex
and let $i \ge 0$ be a non-negative integer.
For an $i$-simplex $\sigma \in Y_i$,
we let $T(\sigma)$ denote the set of all bijections
from the finite set $\{1,\ldots, i+1 \}$ of cardinality
$i+1$ to the set $V(\sigma)$ of vertices of $\sigma$.
The symmetric group $S_{i+1}$ acts on the set 
$\{1,\ldots, i+1 \}$ from the left 
and hence on the set $T(\sigma)$ from the right.
Through this action the set $T(\sigma)$ is
a right $S_{i+1}$-torsor.

We define the set $O(\sigma)$ of orientations of $\sigma$
to be the $\pmone$-torsor 
$O(\sigma) = T(\sigma) \times_{S_{i+1},\sgn} \pmone$ 
which is the push-forward 
of the $S_{i+1}$-torsor $T(\sigma)$ 
with respect to the signature character 
$\sgn: S_{i+1} \to \pmone$.
When $i \ge 1$, the $\pmone$-torsor $O(\sigma)$ is
isomorphic, as a set, to the quotient 
$T(\sigma)/A_{i+1}$ of $T(\sigma)$
by the action of the alternating group 
$A_{i+1} = \Ker\, \sgn \subset S_{i+1}$. 
When $i=0$, the $\pmone$-torsor $O(\sigma)$ 
is isomorphic to the product
$O(\sigma) = T(\sigma) \times \pmone$, on which the group $\pmone$
acts via its natural action on the second factor.

Let $i \ge 1$ and let $\sigma \in Y_i$.
For $v \in V(\sigma)$ let $\sigma_v$ denote the
$(i-1)$-simplex 
$\sigma_v 
= \sigma \times_{V(\sigma)} (V(\sigma) \setminus \{v\})$.
There is a canonical map
$s_v : O(\sigma) \to O(\sigma_v)$ of 
$\pmone$-torsors defined as follows.
Let $\nu \in O(\sigma)$ and take a lift
$\wt{\nu}:\{1,\ldots,i+1\} \xto{\cong} V(\sigma)$
of $\nu$ in $T(\sigma)$. Let 
$\wt{\iota}_v : \{1,\ldots,i\} \inj \{1,\ldots,i+1\}$
denote the unique order-preserving injection
whose image is equal to $\{1,\ldots,i+1\} \setminus 
\{\wt{\nu}^{-1}(v)\}$. It follows from the
definition of $\wt{\iota}_v$ that the composite
$\wt{\nu} \circ \wt{\iota}_v: \{1,\ldots,i\} 
\to V(\sigma)$ induces a bijection
$\wt{\nu}_v : \{1,\ldots,i\}
\xto{\cong} V(\sigma) \setminus \{v\}
= V(\sigma_v)$. We regard $\wt{\nu}_v$
as an element in $T(\sigma_v)$. 
We define $s_v : O(\sigma) \to O(\sigma_v)$
to be the map which sends $\nu \in O(\sigma)$
to $(-1)^{\wt{\nu}^{-1}(v)}$ times
the class of $\wt{\nu}_v$. It is easy to check that
the map $s_v$ is well-defined.

Let $i \ge 2$ and $\sigma \in Y_i$.
Let $v, v' \in V(\sigma)$ with $v \neq v'$.
We have $(\sigma_v)_{v'} = (\sigma_{v'})_v$.
Let us consider the two composites
$s_{v'} \circ s_v :
O(\sigma) \to O((\sigma_v)_{v'})$ and
$s_v \circ s_{v'} :
O(\sigma) \to O((\sigma_{v'})_v)$.
It is easy to check that the equality
\begin{equation} \label{formula1}
s_{v'} \circ s_v (\nu) 
= (-1) \cdot s_v \circ s_{v'} (\nu)
\end{equation}
holds for every $\nu \in O(\sigma)$.

\subsubsection{Cohomology and homology}
\label{sec:def homology}
We say that a simplicial complex
$Y_\bullet$ is locally finite if for any $i \ge 0$
and for any $\tau \in Y_i$, there exist only
finitely many $\sigma \in Y_{i+1}$ such that
$\tau$ is a face of $\sigma$.
We recall the four notions of homology or cohomology
for a locally finite simplicial complex.
Let $Y_\bullet$ be a simplicial complex 
(\resp a locally finite simplicial complex).
For an integer $i\ge 0$, we let 
$Y_i'=\coprod_{\sigma \in Y_i} O(\sigma)$ 
and regard it as a $\pmone$-set.
Given an abelian group $M$, 
we regard the abelian groups
$\bigoplus_{\nu \in Y_i'} M$
and $\prod_{\nu \in Y_i'}M$
as $\pmone$-modules in such a way that
the component at $\eps\cdot \nu$ of
$\eps \cdot (m_\nu)$ is equal to 
$\eps m_\nu$ for $\eps \in \pmone$
and for $\nu \in Y_i'$.

For $i \ge 1$, we let 
$\wt{\partial}_{i,\oplus}: \bigoplus_{\nu \in Y_i'}M
\to \bigoplus_{\nu \in Y_{i-1}'}M$ 
(\resp $\wt{\partial}_{i,\prod}: \prod_{\nu \in Y_i'}M
\to \prod_{\nu \in Y_{i-1}'}M$)
denote the homomorphism of abelian groups 
which sends
$m = (m_\nu)_{\nu \in Y_i'}$ to the element
$\wt{\partial}_i(m)$ whose coordinate 
at $\nu' \in O(\sigma') \subset Y_{i-1}'$ is equal to
\begin{eqnarray}\label{eqn:boundary}
\wt{\partial}_i(m)_{\nu'} =
\sum_{(v,\sigma,\nu)} m_\nu
\end{eqnarray}
where in the sum in the right hand side
$(v,\sigma,\nu)$ runs over the triples
of $v \in Y_0 \setminus V(\sigma')$,
an element $\sigma \in Y_i$, and $\nu \in O(\sigma)$
which satisfy $V(\sigma) = V(\sigma') \amalg \{v\}$ and
$s_v(\nu) = \nu'$.
Note that the sum on 
the right hand side is a finite sum
for $\wt{\partial}_{i,\oplus}$ by definition.
One can see also that the sum is a finite sum
in the case of $\wt{\partial}_{i,\prod}$ 
using the locally finiteness of $Y_\bullet$.
Each of $\wt{\partial}_{i,\oplus}$ 
and $\wt{\partial}_{i,\prod}$ is a
homomorphism of $\pmone$-modules. 
Hence it induces a homomorphism
$\partial_{i,\oplus} : (\bigoplus_{\nu \in Y_i'}M)_\pmone
\to (\bigoplus_{\nu \in Y_{i-1}'}M)_\pmone$
(\resp $\partial_{i,\prod} : (\prod_{\nu \in Y_i'}M)_\pmone
\to (\prod_{\nu \in Y_{i-1}'}M)_\pmone$) of abelian groups,
where the subscript $\pmone$ means the coinvariants.
It follows from the formula (\ref{formula1}) 
and the definition
of $\partial_{i,\oplus}$ and $\partial_{i,\prod}$ that
the family of abelian groups
$((\bigoplus_{\nu \in Y_i'}M)_\pmone)_{i\ge 0}$
(resp.
$((\prod_{\nu \in Y_i'}M)_\pmone)_{i\ge 0}$)
indexed by the integer $i \ge 0$, together with
the homomorphisms $\partial_{i,\oplus}$ 
(resp. $\partial_{i,\prod}$)
for $i \ge 1$,
forms a complex of abelian groups.
The homology groups of this complex are 
the homology groups $H_*(Y_\bullet, M)$
(resp. the Borel-Moore homology groups
$H_*^\BM(Y_\bullet, M)$)
of $Y_\bullet$ 
with coefficients in $M$.
We note that there is a canonical map
$H_*(Y_\bullet, M) \to H_*^\BM(Y_\bullet, M)$
from homology to Borel-Moore homology
induced by the map 
of complexes
$((\bigoplus_{\nu \in Y_i'}M)_\pmone)_{i\ge 0}
\to ((\prod_{\nu \in Y_i'}M)_\pmone)_{i\ge 0}$
given by inclusion at each degree.

The family of abelian groups 
$(\Map_{\pmone}(Y_i', M))_{i\ge 0}$
(resp.
$(\Map^{\mathrm{fin}}_{\pmone}(Y_i', M))_{i\ge 0}$
where the superscript $\mathrm{fin}$ means finite support)
of the $\pmone$-equivariant maps from $Y_i'$ to $M$
forms a complex of
abelian groups in a similar manner.
(One uses the locally finiteness of $Y_\bullet$
for the latter.)  
The cohomology groups of this complex
are the cohomology groups $H^*(Y_\bullet, M)$
(resp. the cohomology groups with compact support 
$H_c^*(Y_\bullet, M)$)
of $Y_\bullet$ 
with coefficients in $M$.

\subsubsection{Universal coefficient theorem}
\label{univ_coeff}
It follows from the definition that
the following universal coefficient
theorem holds.  
That is, for a simplicial complex $Y_\bullet$, 
there exist canonical short exact sequences
$$
0 \to H_i(Y_\bullet, \Z) \otimes M
\to H_i(Y_\bullet, M) \to 
\Tor_1^\Z (H_{i-1}(Y_\bullet, \Z),M) \to 0
$$
and
$$
0 \to \Ext^1_\Z(H_{i-1}(Y_\bullet, \Z),M)
\to H^i(Y_\bullet, M) \to 
\Hom_\Z (H_i(Y_\bullet, \Z),M) \to 0.
$$
for any abelian group $M$.

Similarly, for a locally finite simplicial complex $Y_\bullet$,
we have short exact sequences
$$
0 \to \Ext^1_\Z(H_c^{i+1}(Y_\bullet, \Z),M)
\to H^\BM_i(Y_\bullet,M)
\to \Hom_\Z (H_c^i(Y_\bullet, \Z),M) \to 0
$$
and
$$
0 \to H^i_c (Y_\bullet, \Z) \otimes M
\to H_c^i(Y_\bullet, M) \to 
\Tor_1^\Z (H_c^{i+1}(Y_\bullet, \Z),M) \to 0
$$
for any abelian group $M$. 
The canonical inclusions
$$
\begin{array}{c}
\left( \bigoplus_{\nu \in Y_i'}M \right)_\pmone
\inj \left( \prod_{\nu \in Y_i'}M \right)_\pmone \text{and}
\\
\Map^{\mathrm{fin}}_{\pmone}(Y_i', M)
\inj \Map_{\pmone}(Y_i', M)
\end{array}
$$
for $i \ge 0$ induce 
homomorphisms $H_i(Y_\bullet, M)
\to H^\BM_i(Y_\bullet, M)$ and
$H_c^i(Y_\bullet, M) \to H^i(Y_\bullet, M)$
of abelian groups, respectively.

\subsubsection{}
\label{sec:quasifinite}
Let $f=(f_i)_{i \ge 0}:
Y_\bullet \to Z_\bullet$ be a map of
simplicial complexes. For each integer $i \ge 0$
and for each abelian group $M$,
the map $f$ canonically 
induces homomorphisms $f_* : H_i(Y_\bullet, M)
\to H_i(Z_\bullet,M)$ and
$f^* : H^i(Z_\bullet,M) \to H^i(Y_\bullet,M)$.
We say that the map $f$ is finite
if the subset $f_i^{-1}(\sigma)$ of $Y_i$ is
finite for any $i \ge 0$ and for any 
$\sigma \in Z_i$. If $Y_\bullet$ and $Z_\bullet$ are
locally finite, and if $f$ is finite, then
$f$ canonically 
induces the 
pushforward homomorphism $f_* : H^\BM_i(Y_\bullet, M)
\to H^\BM_i(Z_\bullet,M)$ and
the pullback homomorphism
$f^*: H_c^i(Z_\bullet,M) \to H_c^i(Y_\bullet,M)$.

\subsubsection{} \label{sec:CWcomplex}
 Let $Y_\bullet$ be a simplicial complex.
We associate a CW complex $|Y_\bullet|$ which we call
the geometric realization of $Y_\bullet$.
Let $I(Y_\bullet)$ denote the disjoint
union $\coprod_{i \ge 0} Y_i$. We define a partial order
on the set $I(Y_\bullet)$ by saying that $\tau \le \sigma$
if and only if $\tau$ is a face of $\sigma$.
For $\sigma \in I(Y_\bullet)$, we let $\Delta_\sigma$ denote
the set of maps $f:V(\sigma) \to \R_{\ge 0}$ satisfying
$\sum_{v \in V(\sigma)} f(v) =1$. We regard $\Delta_\sigma$
as a topological space whose topology is induced from that
of the real vector space $\Map(V(\sigma),\R)$.
If $\tau$ is a face of $\sigma$, we regard the space
$\Delta_\tau$ as the closed subspace of $\Delta_\sigma$
which consists of the maps $V(\sigma) \to \R_{\ge 0}$
whose support is contained in the subset 
$V(\tau) \subset V(\sigma)$.
We let $|Y_\bullet|$ denote the colimit
$\varinjlim_{\sigma \in I(Y_\bullet)} \Delta_\sigma$
in the category of topological spaces
and call it the geometric realization of $Y_\bullet$.
%
It follows from the definition that
the geometric realization $|Y_\bullet|$ has a canonical
structure of CW-complex.

\subsubsection{Cellular versus singular}
\label{sec:cellular}
We give a remark on the use of the term 
``Borel-Moore homology''
in this paragraph.
Given a strict simplicial complex,
its cohomology, homology and cohomology with compact support
(for a locally finite strict simplicial complex)
are usually defined as above,
and called cellular (co)homology.  See for example \cite{Hatcher}.

On the other hand, there is also the singular (co)homology
and with compact support that are 
defined using the singular (co)chain complex.  
It is well-known that the cellular (co)homology groups
(with compact support) are isomorphic to the 
singular (co)homology groups (with compact support)
of the geometric realization.
The same proof applies to the generalized simplicial 
complexes and gives an isomorphism between the
cellular and the singular theories.

For the Borel-Moore homology, we did not find a cellular 
definition as above, except in Hattori \cite{Hattori} 
where he does
not call it the Borel-Moore homology.  He also gives a
definition using singular chains and shows that the two homology
groups are isomorphic.

There are several definitions of Borel-Moore homology
that may be associated to a (strict) simplicial complex.
The definition of the Borel-Moore homology for PL manifolds 
is found in Haefliger \cite{Haefliger}.  
There is also a sheaf theoretic definition in Iversen 
\cite{Iversen}.
More importantly, there is the general definition which concerns
the intersection homology.
However, we did not find a statement in the literature
and we did not check that the cellular
definition in Hattori (same as the one given in this article) is
isomorphic to the other Borel-Moore homology theories.


\section{The Bruhat-Tits building and apartments}
\label{sec:BT}
For the general theory of Bruhat-Tits building and
apartments, the reader is referred to the book \cite{Ab-Br}.

\subsection{The Bruhat-Tits building of $\mathrm{PGL}_d$}
In the following paragraphs, 
we recall the definition of
the Bruhat-Tits building of $\PGL_d$ over $K$.
We recall that it is a strict simplicial complex.

\subsubsection{Notation}
Let $K$ be a nonarchimedean local field.
We let $\cO \subset K$ denote the ring of integers.
We fix a uniformizer $\varpi \in \cO$.
Let $d \ge 1$ be an integer.
Let $V=K^{\oplus d}$.  We regard it as the set of row vectors
so that $\GL_d(K)$ acts from the right by multiplication.

\subsubsection{Bruhat-Tits building (\cite{BT})}
\label{Bruhat-Tits}

An $\cO$-lattice in $V$ is 
a free $\cO$-submodule of $V$ of rank $d$.
We denote by $\Lat_{\cO}(V)$ the
set of $\cO$-lattices in $V$.
We regard the set $\Lat_{\cO}(V)$ as
a partially ordered set whose elements are ordered 
by the inclusions of $\cO$-lattices.
%
Two $\cO$-lattices $L$, $L'$ of $V$
are called homothetic if $L = \varpi^j L'$
for some $j \in \Z$. Let $\Latbar_{\cO}(V)$
denote the set of homothety classes of 
$\cO$-lattices $V$. We denote by
$\cl$ the canonical surjection
$\cl: \Lat_{\cO}(V) \to
\Latbar_{\cO}(V)$.
We say that a subset $S$ of 
$\Latbar_{\cO}(V)$
is totally ordered if $\cl^{-1}(S)$ is 
a totally ordered subset of $\Lat_{\cO}(V)$.
The pair $(\Latbar_{\cO}(V), \Delta)$ 
of the set $\Latbar_{\cO}(V)$ 
and the set $\Delta$ of totally ordered finite subsets of
$\Latbar_{\cO}(V)$ forms
a strict simplicial complex.
The Bruhat-Tits building of $\PGL_d$ over $K$
is a simplicial complex $\cBT_\bullet$ which is 
isomorphic to the simplicial complex associated to
this strict simplicial complex.
In the next paragraphs we explicitly describe
the simplicial complex $\cBT_\bullet$.

\subsubsection{}
For an integer $i \ge 0$, let
$\wt{\cBT}_i$ be the set of sequences $(L_j)_{j \in \Z}$
of $\cO$-lattices in $V$ indexed by $j \in \Z$
such that $L_j \supsetneqq L_{j+1}$ 
and $\varpi L_j=L_{j+i+1}$ hold for all $j\in \Z$.
In particular, if $(L_j)_{j \in \Z}$ is an element in 
$\wt{\cBT}_0$, then $L_j = \varpi^j L_0$ for
all $j \in \Z$. We identify the set
$\wt{\cBT}_0$ with the set $\Lat_{\cO}(V)$
by associating the $\cO$-lattice $L_0$
to an element $(L_j)_{j \in \Z}$ in $\cBT_0$.
We say that two elements $(L_j)_{j \in \Z}$
and $(L'_j)_{j \in \Z}$ in $\wt{\cBT}_i$ are equivalent if
there exists an integer $\ell$ satisfying
$L'_j=L_{j+\ell}$ for all $j \in \Z$.
We denote by $\cBT_i$ the set of the equivalence
classes in $\wt{\cBT}_i$. 
For $i=0$, the identification $\wt{\cBT_0}
\cong \Lat_{\cO}(V)$ gives an identification
$\cBT_0 \cong \Latbar_{\cO}(V)$.

Let $\sigma \in \cBT_i$ and take a representative
$(L_j)_{j \in \Z}$ of $\sigma$.
For $j \in \Z$, let us consider the class
$\cl(L_j)$ in $\Latbar_{\cO}(V)$.
Since $\varpi L_j = L_{j+i+1}$, we have
$\cl(L_j) = \cl(L_{j+i+1})$. Since
$L_j \supsetneqq L_k \supsetneqq \varpi L_j$
for $0 \le j < k \le i$, the elements
$\cl(L_0), \ldots, \cl(L_i) \in 
\Latbar_{\cO}(V)$
are distinct. Hence the subset 
$V(\sigma) = \{ \cl(L_j)\ |\ j \in \Z \} \subset \cBT_0$ 
has cardinality $i+1$ and
does not depend on the choice of $(L_j)_{j \in \Z}$.
It is easy to check that the map from
$\cBT_i$ to the set of finite subsets of 
$\Latbar_{\cO}(V)$ which sends
$\sigma \in \cBT_i$ to $V(\sigma)$ is injective
and that
the set $\{V(\sigma)\ |\ \sigma \in \cBT_i\}$
is equal to the set of totally ordered subsets of
$\Latbar_{\cO}(V)$ with cardinality
$i+1$. In particular, for any $j \in \{0,\ldots,i\}$ 
and for any subset $V' \subset V(\sigma)$
of cardinality $j+1$, there exists a unique element
in $\cBT_j$, which we denote by $\sigma \times_{V(\sigma)} V'$,
such that $V(\sigma \times_{V(\sigma)} V')$ is equal to $V'$.
Thus the collection $\cBT_\bullet = \coprod_{i \ge 0} \cBT_i$
together with the data $V(\sigma)$ and
$\sigma \times_{V(\sigma)} V'$
forms a simplicial complex which is canonically isomorphic
to the simplicial complex associated to the
strict simplicial complex 
$(\Latbar_{\cO}(V), \Delta)$ 
which we introduced in the
first paragraph of Section~\ref{Bruhat-Tits}.
We call the simplicial complex $\cBT_\bullet$
the Bruhat-Tits building of $\PGL_d$ over $K$.

\subsubsection{}
\label{sec:dimension}
The simplicial complex $\cBT_\bullet$ is of dimension
at most $d-1$, by which we mean that $\cBT_i$ is an empty
set for $i > d-1$. It follows from the fact that
$\wt{\cBT}_i$ is an empty set for $i > d-1$, which
we can check as follows. Let $i > d-1$ and assume 
that there exists an element 
$(L_j)_{j \in \Z}$ in $\wt{\cBT}_i$.
Then for $j=0,\ldots,i+1$, the quotient 
$L_j/L_{i+1}$ is a subspace of the $d$-dimensional 
$(\cO/\varpi \cO)$-vector 
space $L_0/L_{i+1}=L_0/\varpi L_0$. These subspaces
must satisfy $L_0/L_{i+1} \supsetneqq L_1/L_{i+1}
\supsetneqq \cdots \supsetneqq L_{i+1}/L_{i+1}$.
It is impossible since $i+1 > d$.

\subsection{Apartments}
\label{sec:apartments}
\ 

\

Here we recall the definition of the apartments
which are simplicial subcomplexes of the Bruhat-Tits 
building. We then associate
to each apartment a class in the Borel-Moore homology 
of a quotient of the Bruhat-Tits building.
This class is an analogue of a modular symbol.

\subsubsection{}
\label{sec:def apartment}
We give an explicit description of the simplicial complex
$A_\bullet$ below without making use of the theory 
 of root systems.  
For the viewpoint in the general theory of root systems, 
we refer the reader to \cite[p. 523, 10.1.7 Example]{Ab-Br}.

set $A_0 = \Z^{\oplus d}/\Z(1,\ldots,1)$.
For two elements $x=(x_j), y=(y_j) \in \Z^{\oplus d}$,
we write $x \le y$ when $x_j \le y_j$ for all $1 \le j \le d$.
We say that a subset $\wt{\sigma} \subset \Z^{\oplus d}$
is small if for any two elements $x, y \in \wt{\sigma}$ 
we have either $x \le y \lneq x+(1,\ldots,1)$
or $y \le x \lneq y+(1,\ldots,1)$.
Explicitly, this means that 
$\wt{\sigma}$ is a finite set and is of the form
$\wt{\sigma} = \{x_0,\ldots,x_i \}$ for some elements
$x_0, \ldots, x_i$ satisfying 
$x_0 \lneq \cdots \lneq x_i \lneq x_{i+1} = x_0
+ (1,\ldots,1)$.
We say that a finite subset $\sigma \subset A_0$ has
a small lift to $\Z^{\oplus d}$ if there exists a small subset 
$\wt{\sigma} \subset \Z^{\oplus d}$ which maps bijectively
onto $\sigma$ under the canonical surjection $\Z^{\oplus d} \surj A_0$.
For $i \ge 0$, we let $A_i$ denote the
set of the subsets $\sigma \subset A_0$ with cardinality 
$i+1$ which has a small lift to $\Z^{\oplus d}$.
It is clear that the pair $(A_0, \coprod_{i\ge 0} A_i)$
forms a strict simplicial complex and the 
collection $A_\bullet =(A_i)_{i \ge 0}$
is the simplicial complex associated to the
strict simplicial complex $(A_0, \coprod_{i\ge 0} A_i)$.
We note that $A_i$ is an empty set for $i \ge d$, 
since by definition there is no small subset of $\Z^{\oplus d}$
with cardinality larger than $d$.

\subsubsection{} \label{sec:521}
Let $v_1, \dots, v_d$ be a basis of $V =K^{\oplus d}$.
We define a map $\iota_{v_1,\ldots,v_d} : 
A_\bullet \to \cBT_\bullet$ of simplicial complexes.

Let $\wt{\iota}_{v_1,\ldots,v_d}: \Z^{\oplus d} \to \wt{\cBT}_0$ 
denote the map which sends the element $(n_1,\ldots,n_d) \in \Z^d$
to the $\cO$-lattice 
$\cO\varpi^{n_1} v_1  
\oplus 
\cO \varpi^{n_2} v_2  
\oplus \dots \oplus
\cO \varpi^{n_d} v_d$.
Let $i \ge 0$ be an integer and let $\sigma \in A_i$.
Take a small subset $\wt{\sigma} \subset \Z^{\oplus d}$
with cardinality $i+1$ which maps bijectively onto $\sigma$
under the surjection $\Z^{\oplus d} \to \Z^{\oplus d}/
\Z (1,\ldots,1) = A_0$. By definition the set $\wt{\sigma}$
is of the form $\wt{\sigma} = \{x_0, \ldots, x_i\}$
where $x_0 ,\ldots, x_i \in \Z^{\oplus d}$ satisfy $x_0 \lneq \cdots \lneq x_i \lneq x_{i+1}$
where we have set $x_{i+1} = x_0 + (1,\ldots,1)$.
For each integer $j\in \Z$ we write $j$ in the form
$j=m(i+1) + r$ with $m \in \Z$ and $r \in \{0,\ldots, i\}$,
and set $x_j = x_r + m(1,\ldots,1)$ and $L_j = 
\wt{\iota}_{v_1,\ldots,v_d}(x_j)$.
The sequence $(L_j)_{j \in \Z}$ of $\cO$-lattices
gives an element 
$\wt{\iota}_{v_1,\ldots,v_d,i}(\wt{\sigma})$
in $\wt{\cBT}_i$.
We denote by $\iota_{v_1,\ldots,v_d,i}(\sigma)$
the class of $\wt{\iota}_{v_1,\ldots,v_d,i}(\wt{\sigma})$
in $\cBT_i$.

\begin{lem} \label{lem:iota}
The class $\iota_{v_1,\ldots,v_d,i}(\sigma)$
does not depend on the choice of a small lift~$\wt{\sigma}$.
\end{lem}
\begin{proof}
The inverse image of $\sigma$ under the canonical surjection
$\Z^{\oplus d} \to \Z^{\oplus d}/\Z (1,\ldots,1)$ is equal to
$\{ x_j \ |\ j \in \Z\}$. Since $x_j \lneq x_{j'}$ for
$j \lneq j'$ and $x_{j+i+1} = x_j + (1,\ldots,1)$, any
small subset $\wt{\sigma}'$ of $\Z^{\oplus d}$ with cardinality $i+1$
which maps bijectively onto $\sigma$ is of the form
$\wt{\sigma}' = \{x_{l}, x_{l +1}, \ldots, x_{l +i} \}$ 
for some $l \in \Z$.
The element $\wt{\iota}_{v_1,\ldots,v_d,i}(\wt{\sigma}')$
is the sequence $(L'_j)_{j\in \Z}$, where 
$L'_j = L_{j+l}$. Hence the two elements
$\wt{\iota}_{v_1,\ldots,v_d,i}(\wt{\sigma})$
and $\wt{\iota}_{v_1,\ldots,v_d,i}(\wt{\sigma}')$
give the same element in $\cBT_i$.
 \end{proof}
It is easy to check that the map
$\iota_{v_1,\ldots,v_d,i} : A_i \to \cBT_i$
is injective for every $i \ge 0$ and that
the collection of the map
$\iota_{v_1,\ldots,v_d,i}$ forms 
a map $\iota_{v_1,\ldots,v_d}:A_\bullet
\to \cBT_\bullet$ of simplicial complexes.
We define a simplicial subcomplex 
$A_{v_1, \dots, v_d , \bullet}$ 
of $\cBT_\bullet$ to be the image of the
map $\iota_{v_1,\ldots,v_d}$ so that
$A_{v_1, \dots, v_d , i}$ is the image of
the map $\iota_{v_1,\ldots,v_d,i}$ for each $i \ge 0$.
We call the subcomplex 
$A_{v_1, \dots, v_d , \bullet}$ of $\cBT_\bullet$
the apartment in $\cBT_\bullet$ corresponding to the basis
$v_1,\ldots,v_d$. Since the map
$\iota_{v_1,\ldots,v_d,i}$ is injective for every $i \ge 0$,
the map $\iota_{v_1,\ldots,v_d}$ induces an isomorphism
$A_\bullet \xto{\cong} A_{v_1, \dots, v_d , \bullet}$
of simplicial complexes.

\subsubsection{} \label{sec:fundamental class}
We introduce a special element $\beta$ in
the group $H^\BM_{d-1}(A_\bullet,\Z)$, which is
an analogue of the fundamental class.
Let $\sigma \in A_{d-1}$ and take a small lift
$\wt{\sigma} \subset \Z^{\oplus d}$ to $\Z^{\oplus d}$.
By definition the set
$\wt{\sigma}$ is of the form 
$\wt{\sigma} = \{ x_1, \ldots, x_d\}$
with $x_0 \lneq x_1 \lneq \cdots \lneq x_d$ where
we have set $x_0 = x_d -(1,\ldots,1)$.
It follows from this property that for each
integer $i$ with $1 \le i \le d$ there exists 
a unique integer $w(i)$ with $1 \le w(i) \le d$ 
such that the $w(i)$-th coordinate of $x_i - x_{i-1}$
is equal to $1$ and the other coordinates of
$x_i - x_{i-1}$ are equal to zero.
Since we have $\sum_{i=1}^d (x_i - x_{i-1})
= x_d -x_0 = (1,\ldots,1)$, the map $w:\{1,\ldots,d\}
\to \{1,\ldots,d\}$ is injective.
Hence it defines an element $w$ in the symmetric
group $S_d$. The maps $\{1,\ldots,d\} \to 
A_0 = \Z^{\oplus d} / \Z(1,\ldots,1)$
which sends $i$ to the class of $x_{w^{-1}(i)}$ in $A_0$
gives an element $[\wt{\sigma}]$ in $T(\sigma)$.
\begin{lem}
The element $[\wt{\sigma}] \in T(\sigma)$ does not
depend on the choice of a lift $\wt{\sigma}$.
\end{lem}
\begin{proof}
For each integer $j\in \Z$ we write $j$ of the form
$j=md + r$ with $m \in \Z$ and $r \in \{0,\ldots, d-1\}$
and set $x_j = x_r + m(1,\ldots,1)$.
As we have mentioned in the proof of Lemma~\ref{lem:iota},
The inverse image of $\sigma$ under the canonical surjection
$\Z^{\oplus d} \to \Z^{\oplus d}/\Z (1,\ldots,1)$ is equal to
$\{ x_j \ |\ j \in \Z\}$ and any
small lift $\wt{\sigma}'$ of $\sigma$ to $\Z^{\oplus d}$
is of the form
$\wt{\sigma}' = \{x_{l}, x_{l +1}, \ldots, x_{l +d-1} \}$ 
for some $l \in \Z$. 
For each $i \in \{1,\ldots,d\}$,
the unique integer $j \in \{l,l+1,\ldots, l+d-1 \}$ 
such that the $i$-th coordinate of $x_j - x_{j-1}$
is equal to $1$ and the other coordinates of
$x_j - x_{j-1}$ are equal to zero is congruent
to $w^{-1}(i)$ modulo $d$.
Hence the class of $x_j$ in $A_0$ does not depend
on the choice of a small lift $\wt{\sigma}'$.
This proves the claim.
 \end{proof}

We denote by $[\sigma]$ the class of $[\wt{\sigma}]$
in $O(\sigma)$.
We let $\wt{\beta}$ denote the element
$\wt{\beta} = (\beta_{\nu})_{\nu \in A_{d-1}'}$
in $\prod_{\nu \in A_{d-1}'} \Z$ where $\beta_\nu =1$ if
$\nu = [\sigma]$ for some $\sigma \in A_{d-1}$ and
$\beta_\nu=0$ otherwise.
%
We denote by $\beta$ the class of $\wt{\beta}$
in $(\prod_{\nu \in A_{d-1}'} \Z)_\pmone$.

\begin{prop}
The element $\beta \in (\prod_{\nu \in A_{d-1}'} \Z)_\pmone$ 
is a $(d-1)$-cycle in the chain complex which computes
the Borel-Moore homology of $A_\bullet$.
\end{prop}
\begin{proof}
The assertion is clear for $d=1$ since
the $(d-2)$-nd component of the complex is zero.
Suppose that $d \ge 2$.
Let $\tau$ be an element in $A_{d-2}$.
Take a small lift $\wt{\tau} \subset \Z^{\oplus d}$ 
of $\tau$ to $\Z^{\oplus d}$.
By definition the set
$\wt{\tau}$ is of the form 
$\wt{\tau} = \{ x_1, \ldots, x_d\}$
with $x_0 \lneq x_1 \lneq \cdots \lneq x_{d-1}$ where
we have set $x_0 = x_d -(1,\ldots,1)$.
There is a unique $i \in \{1,\ldots,d-1\}$
such that $x_{i} - x_{i -1}$ has two non-zero
coordinates. There are exactly two elements
in $\Z^{\oplus d}$ which is larger than $x_{i-1}$
and which is smaller than $x_i$. We denote these
two elements by $y_1$ and $y_2$.
We set $\wt{\sigma}_j = \wt{\tau} \cup \{y_j\}$
for $j=1,2$. The sets $\wt{\sigma}_1$,
$\wt{\sigma}_2$ are small subsets of $\Z^{\oplus d}$
of cardinality $d$ and their images
$\sigma_1$, $\sigma_2$ under the surjection
$\Z^{\oplus d} \surj \Z^{\oplus d}/\Z(1,\ldots,1)$
are elements in $A_{d-1}$.
For $j=1,2$, let $w_j$ denote the element
$w$ in the symmetric group $S_d$ which appeared
in the first paragraph of 
Section~\ref{sec:fundamental class} for 
$\sigma = \sigma_j$.
It follows from the definition of $\sigma_j$ that
we have $w_1 = w_2 (i,i+1)$, where $(i,i+1)$ denotes the
transposition of $i$ and $i+1$.
It is easily checked that 
the set of the elements in $A_{d-1}$
which has $\tau$ as a face is equal to
$\{\sigma_1,\sigma_2\}$. 
Since we have $\sgn(w_1) = - \sgn(w_2)$,
it follows that the component 
in $(\prod_{\nu \in O(\tau)} \Z)_\pmone$ of the
image of $\beta$ under the boundary map 
$(\prod_{\nu \in A_{d-1}'} \Z)_\pmone 
\to (\prod_{\nu' \in A_{d-2}'} \Z)_\pmone$
is equal
to zero. This proves the claim.
 \end{proof}

\section{Arithmetic groups and modular symbols}
\subsection{Arithmetic groups}\label{sec:71}

\subsubsection{An arithmetic group}
\label{sec:4.1.1}
We give here the definition of 
our main object of study, an arithmetic 
group $\Gamma$.

Let us give the setup.
We let $F$ denote a global field of positive characteristic.
Let $C$ be a proper smooth curve over a finite field
whose function field is $F$.
Let $\infty$ be a place of $F$ and let $K=F_\infty$ denote
the local field at $\infty$.
We let $A=H^0(C \setminus \{\infty\}, \cO_C)$.
Here we identified a closed point of $C$ and
a place of $F$.
We write $\wh{A}=\varprojlim_I A/I$, where
the limit is taken over the nonzero ideals of $A$.
We let  $\A^\infty=\wh{A} \otimes_A F$ denote the ring of 
finite adeles.

Let $\bK^\infty \subset \GL_d(\A^\infty)$ be a compact open subgroup.
We set $\Gamma=\GL_d(F) \cap \bK^\infty$
and regard it as a subgroup of $\GL_d(K)$.
We refer to the group of this form for some $\bK^\infty$
an arithmetic (sub)group of $\GL_d(K)$ (contained in $\GL_d(F)$).

We give a remark.  Let $\Gamma$ be an arithmetic group.
Then $\Gamma \cap \mathrm{SL}_d(F)=\Gamma \cap \mathrm{SL}_d(K)$
is a subgroup of $\Gamma$ of finite index, and is an $S$-arithmetic group
of $\mathrm{SL}_d$ over $F$ for $S=\{\infty\}$ 
dealt in the paper of Harder \cite{Harder}.
\subsubsection{}
\label{sec:def Gamma}
Let $\Gamma \subset \GL_d(K)$ be a subgroup.
We consider the following Conditions (1) to (5) 
on $\Gamma$.
\begin{enumerate}
\item $\Gamma \subset \GL_d(K)$ is a discrete subgroup,
\item $\{\det(\gamma) \,|\, \gamma \in \Gamma \} \subset O_\infty^\times$
where $O_\infty$ is the ring of integers of $K$,
\item $\Gamma \cap Z(\GL_d(K))$ is finite.
\end{enumerate}
Let $A_\bullet=A_{v_1,\dots,v_d,\bullet}$ 
denote the apartment corresponding to a basis 
$v_1,\dots, v_d\in K^{\oplus d}$
(defined in Section~\ref{sec:521}).
\begin{enumerate}
\setcounter{enumi}{3} 
\item For any apartment 
$A_\bullet=A_{v_1,\dots, v_d,\bullet}$
with $v_1, \dots, v_d \in F^{\oplus d}$, 
the composition 
$A_\bullet \hookrightarrow \cB\cT_\bullet \to \Gamma \backslash \cB\cT_\bullet$ 
is quasi-finite,
that is, the inverse image of any simplex by this map is a finite set.
\item The cohomology group $H^{d-1}(\Gamma, \Q)$
is a finite dimensional $\Q$-vector space.
\end{enumerate}
The condition (1) will be used in Lemma \ref{AF}.
The condition (2) implies that each element in the isotropy group of a simplex fixes the vertices of the simplex.
Under the condition (1), 
the condition (3) implies that the stabilizer of a simplex is finite.
This implies that the $\Q$-coefficient group homology of $\Gamma$ 
and the homology of $\Gamma \backslash |\cB\cT_\bullet|$ are isomorphic.
The condition (4) will be used to define a class in Borel-Moore homology 
of $\Gamma\backslash\cB\cT_\bullet$ starting from an apartment 
(Section \ref{sec:def modular symbol}).  
The condition (5) will be used in the proof of Lemma~\ref{lem:5.6}.

Let us show that all five 
conditions of Section~\ref{sec:def Gamma}
are satisfied when $\Gamma$ is an arithmetic subgroup.
The condition (1) holds trivially.
We note that there exists an element $g \in \GL_d(\A^\infty)$
such that $g \bK^\infty g^{-1} \subset \GL_d(\wh{A})$.
Since $\det(\gamma) \in F^\times \cap \wh{A}^\times \subset O_\infty^\times$ for $\gamma \in \Gamma$, (2) holds.
Because $F^\times \cap \GL_d(\wh{A})$ is finite, (3) holds.
\begin{lem}
\label{lem:arithmetic}
Let $\Gamma$ be an arithmetic subgroup. 
Then (4) holds.
\end{lem}

\begin{proof}
We show that the inverse image of each simplex 
of $\Gamma\backslash \cB\cT_\bullet$ 
under the map in (4) is finite.
The set of simplices of $\cB\cT_\bullet$ is 
of a fixed dimension is identified 
(see Section~\ref{sec:6.1.2} for the 
identification) with 
the coset $\GL_d(K)/\wt{\bK}_\infty$
for an open subgroup
$\wt{\bK}_\infty \subset
\GL_d(K)$
which contains
$K^\times \bK_\infty$
as a subgroup of finite index
for some compact open subgroup
$\bK_\infty \subset \GL_d(K)$.
Let $T \subset \GL_d$ denote the diagonal maximal torus.

The set of simplices of 
$A_\bullet$ of fixed dimension is identified
with the image of the map
\[
\coprod_{w \in S_d}
gwT(K)
\to
\GL_d(K)/\wt{\bK}_\infty
\]
for some $g \in \GL_d(F)$.
Since $S_d$ is a finite group, it then suffices to show that for any
$w \in S_d$, the map
$$
\mathrm{Image}[gwT(K) \to \GL_d(K)/K^\times \bK_\infty ]
\to \Gamma \backslash \GL_d(K) /K^\times \bK_\infty
$$
is quasi-finite.
The inverse image under the last map of the image of 
$gwt \in gwT(K)$ 
is isomorphic to
the set
\[
\begin{array}{ll}
\{\gamma \in\Gamma\,|\, 
\gamma gwt \in gwT(K) 
K^\times \bK_\infty\}
&=\Gamma \cap 
gwT(K) K^\times
\bK_\infty (gwt)^{-1}\\
&=\Gamma \cap 
(gw)T(K) t \bK_\infty
t^{-1}
(gw)^{-1}.
\end{array}
\]
Hence, if we let $g'=gw$
and $\bK'_\infty
=t\bK_\infty t^{-1}$,
this set equals
\[
\begin{array}{ll}
\Gamma\cap 
g'T(K) \bK'_\infty g'^{-1}
&
=\GL_d(F) \cap 
(\bK^\infty \times 
g'T(K) \bK'_\infty g'^{-1})
\\
&=
g'(\GL_d(F) \cap (g'^{-1}
\bK^\infty
g' \cap T(K)
\bK'_\infty)
g'^{-1}.
\end{array}
\]
The finiteness is proved in the following lemma.
\end{proof}

\begin{lem}
For any compact open subgroup $\bK \subset \GL_d(\A)$, 
the set $\GL_d(F) \cap T(K) \bK$ is finite.
\end{lem}

\begin{proof}
Let $U=T(K) \cap \bK$.
Then $T(O_\infty) \supset U$ and is of finite index.
Note that there exist a non-zero ideal 
$I\subset A$ and an integer $N$ such that 
$\bK \subset I^{-1}\varpi_\infty^{-N} 
\mathrm{Mat}_d(\wh{A})\times \mathrm{Mat}_d(O_\infty)$
where $\varpi_\infty$ is a uniformizer in $O_\infty$.

Let $\alpha:T(K)/U \to T(K)/T(O_\infty) \cong \Z^{\oplus d}$ 
be the (quasi-finite) map induced by the 
inclusion $U \subset T(O_\infty)$.
For $h \in T(K)$, we write
$(h_1,\dots, h_d)=\alpha(h)$.
Then for $i=1,\ldots,d$, the 
$i$-th row of $h \bK$ is contained in
$(I^{-1}\wh{A} \times \varpi_\infty^{-N}\varpi_\infty^{h_i}O_\infty)^{\oplus d}$.
Hence, for sufficiently large $h_i$, 
the intersection  $h \bK \cap \GL_d(F)$ is empty.  
We then have, for sufficiently large $N'$,

\begin{equation}\label{lem q-finite}
\GL_d(F)\cap T(K) \bK 
=\coprod_{h\in T(K)/U, \atop 
h_1, \ldots, h_d \le N'} 
\GL_d(F) \cap h\bK.
\end{equation}

The adelic norm of the determinant of an element in $\GL_d(F)$ is 1, 
while that of an element in $h \bK$ is 
$|\det h\,|_\infty = \sum_{i=1}^d h_i$.
So (\ref{lem q-finite}) equals
\[\displaystyle 
\coprod_{h\in T(K)/U, h_i \le N', \sum h_i=0} 
\GL_d(F) \cap h \bK.
\]
The index set of the disjoint union above is finite 
since $\alpha$ is quasi-finite,
and $\GL_d(F) \cap h \bK$ is finite
since $\GL_d(F)$ is discrete and 
$h \bK$ is compact.
The claim follows.
 \end{proof}

\begin{lem}
\label{lem:Harder}
Let $\Gamma$ be an arithmetic subgroup.  
Then (5) holds.
\end{lem}
\begin{proof}
This follows from \cite[p.136, Satz 2]{Harder}.
 \end{proof}

\subsection{Arithmetic quotient of the Bruhat-Tits building}
\label{sec:explicit beta2}
Let us define simplicial complex $\Gamma \backslash \cBT_\bullet$
for an arithmetic subgroup $\Gamma$
in this section.

We need a lemma.
\begin{lem} \label{lem:stabilizer}
Let $i \ge 0$ be an integer, let
$\sigma \in \cBT_i$ and let
$v,v' \in V(\sigma)$ be two vertices with
$v \neq v'$. Suppose that
an element $g \in \GL_d(K)$ satisfies
$|\det\, g|_\infty =1$. Then we have
$gv \neq v'$.
\end{lem}
\begin{proof}
Let $\wt{\sigma}$ be an element
$(L_j)_{j \in \Z}$ in $\wt{\cBT}_i$ 
such that the class of $\wt{\sigma}$ in $\cBT_i$ is equal
to $\sigma$. 
There exist two integers $j,j' \in \Z$
such that $v$, $v'$ is the class of $L_j$, $L_{j'}$,
respectively.
Assume that $g v =v'$. Then there exists an integer
$k \in \Z$ such that $L_j g^{-1}
= \varpi_\infty^{k} L_{j'} = L_{j' + (i+1) k}$.
Let us fix a Haar measure $d\mu$ of the $K$-vector space
$V_\infty=K^{\oplus d}$. 
As is well-known, the push-forward of $d\mu$
with respect to the automorphism $V_\infty \to V_\infty$
given by the right multiplication by $\gamma$ is equal
to $|\det\, \gamma|_\infty^{-1} d\mu$ 
for every $\gamma \in \GL_d(K)$.
Since $|\det\, g|_\infty =1$, it follows from the
equality $L_j g^{-1} = L_{j'+(i+1)k}$ that
the two $\cO_\infty$-lattices $L_j$ and $L_{j'+(i+1)k}$ have
a same volume with respect to $d\mu$.
Hence we have $j=j'+(i+1)k$, which implies 
$L_j = \varpi_\infty^k L_{j'}$. It follows that the class of
$L_j$ in $\cBT_0$ is equal to the class of $L_{j'}$,
which contradicts the assumption $v \neq v'$.
 \end{proof}

Let 
$\Gamma \subset \GL_d(K)$
be an arithmetic subgroup.

It follows from Lemma~\ref{lem:stabilizer}
(using Condition (2) of Section~\ref{sec:def Gamma})
that for each $i \ge 0$
and for each $\sigma \in \cBT_i$,
the image of $V(\sigma)$ under the 
surjection $\cBT_0 \surj \Gamma \bsl \cBT_0$
is a subset of $\Gamma \bsl \cBT_0$
with cardinality $i+1$.
We denote this subset by $V(\cl(\sigma))$, since
it is easily checked that it depends only on the class 
$\cl(\sigma)$ of $\sigma$ in $\Gamma \bsl \cBT_i$.
Thus the collection
$\Gamma \bsl \cBT_\bullet =(\Gamma \bsl \cBT_i)_{i \ge 0}$ 
has a canonical structure of a simplicial complex such that
the collection of the canonical surjection
$\cBT_i \surj \Gamma \bsl \cBT_i$ 
is a map of simplicial complexes 
$\cBT_\bullet \surj \Gamma \bsl \cBT_\bullet$.

\subsection{Modular symbols} 
\label{sec:def modular symbol}
Let $v_1,\dots, v_d$ be an $F$-basis 
(that is, a basis of $F^{\oplus d}$
regarded as a basis of $K^{\oplus d}$).
We consider the composite
\begin{equation} \label{quasifinite}
A_\bullet \xto{\iota_{v_1,\ldots,v_d}}
\cBT_\bullet \to \Gamma \bsl \cBT_\bullet.
\end{equation}
Condition (4) implies that 
the map (\ref{quasifinite}) 
is a finite map of simplicial complexes
in the sense of Section~\ref{sec:quasifinite}.
It follows 
that the map (\ref{quasifinite}) induces a homomorphism
$$
H^\BM_{d-1}(A_\bullet, \Z) \to
H^\BM_{d-1}(\Gamma \bsl \cBT_\bullet, \Z).
$$
We let $\beta_{v_1,\ldots,v_d}
\in H^\BM_{d-1}(\Gamma \bsl \cBT_\bullet, \Z)$ denote the
image under this homomorphism 
of the element $\beta \in H^\BM_{d-1}(A_\bullet, \Z)$ introduced
in Section~\ref{sec:fundamental class}.
We call this the class of the apartment
$A_{v_1,\dots, v_d,\bullet}$.

\subsection{Main Theorem}
We are ready to state our theorem.
\begin{thm}\label{lem:apartment}
Let $\Gamma \subset \GL_d(K)$
be an arithmetic subgroup.
The image of the canonical map (see Section~\ref{sec:def homology}
for the definition)
\[
H_{d-1}(\Gamma \backslash \cBT_\bullet, \Q) \to 
H^\BM_{d-1}(\Gamma \backslash \cBT_\bullet, \Q) 
\]
is contained in the sub $\Q$-vector space 
generated by the classes of apartments
associated to $F$-bases.
\end{thm}

\section{Proof of Theorem~\ref{lem:apartment}}\label{Modular Symbols}
The purpose of this section is to prove 
Theorem~\ref{lem:apartment}.
It gives the description of the homology of certain arithmetic groups
in terms of the subspace generated by the classes of modular symbols 
inside the Borel-Moore homology
of the quotient of the Bruhat-Tits building.
The treatment of the modular symbols differs from the archimedean case 
(see \cite{AR})
in that the group does not act freely on the compactification
and in that an apartment is contractible as a subspace of the building.
To compare, we use equivariant homology (Section \ref{Equivariant homology}) of 
Werner's compactification (Section \ref{Werner's compactification}) 
as an intermediary object.

\subsection{Equivariant homology}\label{Equivariant homology}		
Let $\Gamma \subset \GL_d(K)$
be an arithmetic subgroup.
We define the simplicial set 
(not a simplicial complex)
$E\Gamma_\bullet$ as follows.
We define $E\Gamma_n=\Gamma^{n+1}$ to be the $(n+1)$-fold direct product of $\Gamma$ for $n\ge 0$.
The set $\Gamma^{n+1}$ is naturally regarded as the set of maps of sets
$\Map(\{0,\dots, n\}, \Gamma)$ and from this one obtains naturally the 
structure of a simplicial set.
We let $|E\Gamma_\bullet|$ denote the geometric realization of $E\Gamma_\bullet$.  
Then $|E\Gamma_\bullet|$ is contractible.
We let $\Gamma$ act diagonally on each $E\Gamma_n$ $(n\ge 0)$.  
The induced action on $|E\Gamma_\bullet|$ is free.

Let $M$ be a topological space on which $\Gamma$ acts.
The diagonal action of $\Gamma$ on $M\times |E\Gamma_\bullet|$ is free.
We let $H_*^\Gamma(M, B)= H_*(\Gamma\backslash(M \times |E\Gamma_\bullet|), B)$ where $B$ is a coefficient ring,
and call it the equivariant homology of $M$ with coefficients in $B$.  
We also use the relative version, 
and define equivariant cohomology in a similar manner. 

\subsection{Werner's compactification}\label{Werner's compactification}
In this section, we briefly recall the result of Werner (\cite{We2}, \cite{We1}).
\subsubsection{Semi-norms}
Let $W$ be an $K$-vector space.
We call a function $\gamma:W \to \R_{\ge 0}$
a semi-norm if the following conditions are satisfied:
\begin{enumerate}
\item $\gamma(\lambda w)=|\lambda| \gamma(w)$ for $\lambda \in F_\infty, w\in W$, 
\item $\lambda(w_1+w_2) \le \sup\{\gamma(w_1), \gamma(w_2)\}$ for $w_1,w_2 \in W$, 
\item There exists an element $w \in W$ 
satisfying $\gamma(w) \neq 0$.
\end{enumerate}
We say that two semi-norms are equivalent if and only if
one is a non-zero constant multiple of the other.

Let $V^* = \Hom_{K}(V,K)$
be the dual vector space of $V$.
We endow the set $S'$ of semi-norms on $V^*$ 
with the topology of pointwise convergence.
We give the set $S$ of equivalence classes of semi-norms 
the quotient topology.

%
\subsubsection{}
We write $\BTbar$ for the compactification of $\BT$ of Werner in \cite{We2}
(which uses lattices of smaller rank),
and let $\partial\BTbar=\BTbar\setminus \BT$.
The topological space $\BTbar$ is compact and contractible
(\cite[p.519, Theorem 4.1]{We2}).
The action of $\GL(V)$ is extended to $\BTbar$ (\cite[Theorem 4.2]{We1}).

By a theorem of Goldman-Iwahori (see \cite[Theorem 2.2]{De-Hu}), 
the set of equivalence classes of norms on $V^*$ is isomorphic to the set of points of 
the geometric realization of the Bruhat-Tits building for $\PGL(V^*)$.
In the paper of Werner \cite[Theorem 5.1]{We1}, this isomorphism is
extended to a canonical homeomorphism 
$S \cong \overline{|\cB\cT_{V^*,\bullet}|}'$
where $\overline{|\cB\cT_{V^*,\bullet}|}'$ is the compactification of $|\cB\cT_{V^*,\bullet}|$ 
using semi-norms.
We use the homeomorphism 
$\overline{|\cB\cT_{V^*,\bullet}|}' \cong \overline{|\cB\cT_{V,\bullet}|}$
of Werner (\cite{We1} p.518), and obtain a homeomorphism $S \cong \overline{|\cB\cT_{V,\bullet}|}$.


\subsection{}
Let us give an outline of the proof of Theorem~\ref{lem:apartment}
in this section.
We construct the following commutative 
diagram in Sections~\ref{sec:4 and 5} and~\ref{sec:2 and 3}:
\begin{equation}
\label{eqn:diagram0}
\xymatrix{
H_{d-1}(\GBT,\Q)
\ar[r]^{(1)}
&
H_{d-1}^\BM(\GBT,\Q)
\\
H_{d-1}^\Gamma(\cgBTdot,\Q)
\ar[u]^{(2)}_\cong
\ar[r]^{(5)\phantom{a;lskj}}
&
H_{d-1}^\Gamma
(\cgBTdot, \partial \cgBTdot; \Q)
\ar[u]^{(3)}
\\
H_{d-1}(\Gamma,\Q)
\ar[u]^{(4)}_{\cong\phantom{;lkasdjf;lk}}
}
\end{equation}
Here the map (1) is the map that appeared in the statement
of Theorem~\ref{lem:apartment}.
The map (5) is the pushforward map of homology.
The other maps will be constructed later.
It is easy to see that the groups
$H_{d-1}(\Gamma,\Q)$ and 
$H_{d-1}(\GgbarBTEG,\Q)$
are isomorphic since $\cgBTdot \times |E \Gamma_\bullet|$
is contractible and $\Gamma$ acts freely.
However, the key here is to construct (4) explicitly 
at the level of chain complexes in the direction indicated 
by the arrow above, so that we are able to compute
explicitly the image of the composite $(5)\circ (4)$.

The construction of the square is elementary, but there
is one problem which is caused by that
the isomorphism
between the Borel-Moore homology as defined in this
paper and the Borel-Moore homology of a topological
space in general is not found in the literature.
We resort to the well-known cases of the isomorphisms
for homology and for cohomology to circumvent this problem.
In order to do so, we use property (5) of the arithmetic 
group (Section~\ref{sec:def Gamma}) 
and take the dual twice.

\subsection{On the maps (4) and (5)}
\label{sec:4 and 5}
We construct the map (4) and show that it is an isomorphism.
This is done very explicitly, so that 
we are able to compute the image of the 
composite map
(5)(4) (Lemma~\ref{AF}).

\subsubsection{}
We let $C_\bullet$ denote the complex of $\Z[\Gamma]$-modules 
defined by $C_n=\Z[\Gamma^{n+1}] \,\,(n\ge 0)$
and the usual
boundary homomorphisms.
It is a free resolution of the trivial $\Z[\Gamma]$-module $\Z$.
The homology group of the $\Gamma$-coinvariants $C_{\Gamma,\bullet}$ of $C_\bullet$
is the group homology $H_*(\Gamma,\Z)$.

Let $D_n=\Z[\mathrm{Map}_{\mathrm{cont}}(\Delta_n, \overline{|\cB\cT_\bullet|}\times |E\Gamma_\bullet|)]$.
The usual boundary map 
turns $D_\bullet$ into a complex of $\Z[\Gamma]$-modules 
with $\Gamma$ acting on $\XbartimesEGamma$ diagonally.
It is a free resolution of the trivial $\Z[\Gamma]$-module
since the action of $\Gamma$ on $\XbartimesEGamma$ is free.

The $\Gamma$-coinvariants, denoted $D_{\Gamma,\bullet}$, is canonically
isomorphic to the module 
\[
\Z[\mathrm{Map}_{\mathrm{cont}}(\Delta_n, \Gamma \backslash \XbartimesEGamma)].
\]
Hence the homology group of the complex $D_{\Gamma,\bullet}$ is  $H_*^\Gamma(\overline{|\cB\cT_\bullet|},\Z)$.

Let $r \ge 0$ be an integer. Let 
$\Delta_r=\{(t_0,\dots, t_r) \in \R^{r+1}| \sum t_i=1, 0\le t_i \le 1\}$ 
be the (geometric) $r$-simplex.
Given $v_0,\dots, v_r \in V\setminus \{0\}$,
we construct a map $s(v_0,\dots, v_r):\Delta_r \to S \cong \overline{|\cB\cT_\bullet|}$ as follows.
For $(t_0,\dots, t_r)\in \Delta_r$ and $f \in V^*$, 
we set
\[
s(v_0,\dots, v_r)(t_0,\dots, t_r)(f)=
\sup_{0 \le i \le r}
|f(v_i)| q_\infty^{-1/t_i}.
\]
Here,  we set $q_\infty^{-1/t_i}=0$ if $t_i=0$.
It is easy to check that 
$s(v_0,\dots, v_r)(t_0,\dots, t_r)$ is a semi-norm
on $V^*$
for each $(t_0,\dots, t_r)\in \Delta_r$.

\begin{lem}
\label{lem:cont s}
The map $s(v_0,\dots, v_r)$ is continuous.
\end{lem}
\begin{proof}
This is immediate from the definition of
the topology on $S'$ and on $S$, since for each $f \in V^*$, 
the map $s(v_0,\dots, v_r)(t_0,\dots, t_r)(f): 
\Delta_r \to \R_{\ge 0}$ is continuous.
 \end{proof}

\subsubsection{}
Let $r \ge 0$ be an integer.
Given $v\in V\setminus\{0\}$ and $[g_0,\dots,g_r]\in \Z[\Gamma^{r+1}]=C_r$,
we set $\eta_v([g_0,\dots, g_r])=
s(g_0 v,\dots,g_r v) \times [g_0,\dots,g_r]
: \Delta_r \to \overline{|\cB\cT_\bullet|}
\times |E\Gamma_{\bullet}|$.
Here $[g_0,\dots,g_r]$ on the right hand side is regarded as 
the canonical inclusion 
$\Delta_r \inj |E\Gamma_\bullet|$ associated to
the $r$-simplex $(g_0,\dots, g_r)$ of $E\Gamma_\bullet$.
Extending this by linearity, we obtain a map of complexes 
$\eta_v:C_\bullet \to D_\bullet$.

\begin{lem}
The map $\eta_v$ is a quasi-isomorphism. 
\end{lem}
\begin{proof}
We have seen that both $C_\bullet$ and $D_\bullet$ are free $\Z[\Gamma]$-resolutions
of the trivial $\Z[\Gamma]$-module $\Z$.   So we only need to check at degree 0,
that is, the commutativity of the following diagram:
\[
\begin{CD}
C_0 @>>{\eta_v}> D_0  \\
@VVV        @VVV      \\
\Z @>>{\id}>   \Z
\end{CD}
\]
where the vertical homomorphisms are augmentations. This is clear.
 \end{proof}

Taking $\Gamma$-coinvariants, we obtain a map of complexes $C_{\Gamma,\bullet} \to D_{\Gamma,\bullet}$.
It induces a map of homology $H_*(\Gamma,\Z) \to H_*^\Gamma(\overline{|\cB\cT_\bullet|},\Z)$.
This is an isomorphism since both $C_\bullet$ and $D_\bullet$ are free $\Z[\Gamma]$-resolutions of $\Z$.
We define the map (4) to be this map tensored by $\Q$.

\subsubsection{}
For $n\ge 0$, let 
$$
\wt{D}_{\Gamma,n}=
\Z[\mathrm{Map}_{\mathrm{cont}}(\Delta_n,\Gamma \backslash \XbartimesEGamma)]
/\Z[\mathrm{Map}_{\mathrm{cont}}
(\Delta_n,\Gamma \backslash \partial \XbartimesEGamma)].
$$
Then 
$H_*^\Gamma(\overline{|\cB\cT_\bullet|},\partial\overline{|\cB\cT_\bullet|};\Z)$ 
is the homology group of the 
complex $\wt{D}_{\Gamma,\bullet}$.
The canonical surjection at each degree induces 
a map of complexes $D_{\Gamma,\bullet}\to \wt{D}_{\Gamma,\bullet}$,
and a homomorphism 
$H_*^\Gamma(\overline{|\cB\cT_\bullet|},\Z) \to 
H_*^\Gamma(\overline{|\cB\cT_\bullet|}, 
\partial\overline{|\cB\cT_\bullet|};\Z)$.
The map (5) of the diagram~\eqref{eqn:diagram0}
is this map tensored by $\Q$.

Let $v_1,\dots, v_d\in V$ be a basis and let $g_0,\dots, g_{d-1}\in \Gamma$.
By construction, the image of the faces of $\Delta_{d-1}$ by the continuous map
\[ s(v_1,\dots, v_d) \times [g_0,\dots, g_{d-1}]:\Delta_{d-1} \to \XbartimesEGamma \]
is contained in $\partial\XbartimesEGamma$.
We let $A_{v_1,\dots,v_d;g_0,\dots,g_{d-1}}$ denote the class of this continuous function in $\wt{D}_{\Gamma,d-1}$
and in $H_{d-1}^\Gamma(\overline{|\cB\cT_\bullet|}, \partial \overline{|\cB\cT_\bullet|};\Z)$.
We let $\cA_F^\mathrm{rel}$ denote the submodule of $H_\Gamma^{d-1}(\overline{|\cB\cT_\bullet|}, \partial\overline{|\cB\cT_\bullet|};\Z)$
generated by elements of the form $A_{v_1,\dots, v_d;g_0,\dots, g_{d-1}}$ with
$g_i \in \Gamma \,\, (0\le i \le d-1)$ and 
$v_1,\dots, v_d \in V_F=F^{\oplus d} \subset K^{\oplus d}=V$ 
an $F$-basis.

\begin{lem}\label{AF}
The image of 
\[
H_{d-1}(\Gamma, \Q) 
\xto{(4)} 
H_{d-1}^\Gamma(\overline{|\cB\cT_\bullet|}, \Q)
\xto{(5)} 
H_{d-1}^\Gamma(\overline{|\cB\cT_\bullet|},
\partial\overline{|\cB\cT_\bullet|}; \Q)
\]
is contained in the sub 
$\Q$-vector space generated by $\cA_F^\mathrm{rel}$.
\end{lem}
\begin{proof}
Take a $v \in V_F \backslash \{0\} \subset V\backslash \{0\}$.
Consider the map of complexes 
$C_{\Gamma,\bullet} \to D_{\Gamma, \bullet} \to \wt{D}_{\Gamma,\bullet}$
where the first map is $\eta_v$ and the second map is the canonical map.
The image of $C_{\Gamma,d-1}$ is of the form 
\[ s(g_0v,\dots, g_{d-1}v)\times [g_0,\dots,g_{d-1}] \]
for some $g_0, \dots, g_{d-1} \in \Gamma$.
Since $v \in F^{\oplus d}$ and $g_0,\dots, g_{d-1} \in \Gamma \subset \GL_d(F)$ by the condition (1) on $\Gamma$,
the vectors $g_0 v, \dots, g_{d-1}v$ are $F$-vectors.
If $g_0v,\dots, g_{d-1}v$ do not form a basis,
then the element above is zero in $H_{d-1}^\Gamma(\BTbar,\partial\BTbar;\Q)$
because by the construction of $s$ the image of the map above is contained in 
$\Gamma \backslash \partial\BTbar\times |E\Gamma_\bullet|$.
 \end{proof}


\subsection{The maps (2) and (3)}
\label{sec:2 and 3}
Given a $\Q$-vector space $A$, 
let $A^*=\Hom(A, \Q)$ denote the dual.
In what follows, the coefficient ring is $\Q$
unless otherwise specified.

\subsubsection{}
Consider the following diagram.
\begin{equation}
\label{eqn:diagram1}
\xymatrix{
H_{d-1}(\GBT)
\ar[d]^{(6)}_\cong
\ar[r]^{(1)}
&
H_{d-1}^\BM(\GBT)
\ar[dd]_\cong^{(7)}
\\
H_{d-1}(\GBT)^{**}
\ar[d]_=^{(8)}
\\
H^{d-1}(\GBT)^*
\ar[r]^{(9)}
&
H_c^{d-1}(\GBT)^*
}
\end{equation}
The map (1) is the canonical map from
homology to Borel-Moore homology.
The map (6) is the canonical map 
$A\to A^{**}$ for
$A=H_{d-1}(\GBT)$.
We will see later (Corollary~\ref{cor:5.7})
that 
(6) is
an isomorphism.
The map (7) is the isomorphism 
given by the map in the universal 
coefficient theorem 
(see Section~\ref{univ_coeff}).
The map (8) 
is the dual of the fact that 
cohomology is the dual of homology.
The map (9) 
is the dual of the 
canonical map from cohomology with compact support 
to cohomology.  It follows from the definitions 
that the diagram is commutative.

\subsubsection{}
Consider the diagram
\begin{equation}
\label{eqn:diagram2}
\xymatrix{
H^{d-1}(\GBT)
&
H_c^{d-1}(\GBT)
\ar[l]^{(9)'}
\ar[d]^{(11)}_\cong
\\
&
\varinjlim_L H^{d-1}(\GBT, L)
\\
H^{d-1}(\GgBT)
\ar[uu]^\cong_{(10)}
&
\varinjlim_L
H^{d-1}(\GgBT, |L|),
\ar[l]^{(13)\phantom{;al;lkj}}
\ar[u]_{(12)}^\cong
}
\end{equation}
where $L$ runs over the subsimplicial
complexes of $\GBT$ such that
the complement $|\GBT| \setminus |L|$
is covered by a finite number of simplices.

The map (9)' is the forget support map,
whose dual is the map (9) in diagram \eqref{eqn:diagram1}.
The map (10) is the canonical map
from singular cohomology to cellular cohomology
(see Section~\ref{sec:cellular})
The map (11) is obtained from the definition.
The map (12) at each stage is the 
canonical map from singular cohomology 
to cellular cohomology.
The map (13) is the limit of the pullback maps.
It is easy to check that the diagram is commutative.
\begin{lem}
The map (11) is an isomorphism.
\end{lem}
\begin{proof}
It suffices to show that, 
given a finite set $B$ of simplices of 
$\GBT$,
there exists a subsimplicial complex
$L \subset \GBT$
such that
\begin{enumerate}
\item 
the cardinality of the set of simplices 
not contained in $L$ is finite, that is,
the cardinality of 
$((\GBT) \setminus L)=\cup_{i \ge 0} ((\GBT)_i \setminus L_i)$
is finite, and
\item $B \subset ((\GBT) \setminus L)$. 
\end{enumerate}
Let us construct such an $L$.
Let $\overline{B}$ denote the set of simplices $\sigma$ 
of $\GBT$ such that 
\begin{itemize}
\item there exists a simplex $\tau \in B$ such that
$\sigma$ and $\tau$ has a face in common.
\end{itemize}
Now we set $L$ to be the set of 
simplices $\sigma \in \GBT$ such that
\begin{itemize}
\item 
$\sigma \notin \overline{B}$, or
\item
there exists a simplex $\tau \in ((\GBT) \setminus \overline{B})$
such that $\sigma$ is a face of $\tau$.
\end{itemize}
Then $L$ has a structure of a subsimplicial complex.
It is easy to see that 
\[
B \subset (\GBT \setminus L) \subset \overline{B},
\]
which implies (2) above. 
Since $\GBT$ is locally finite,
$\overline{B}$ is a finite set,
which implies (1).
 \end{proof}

\subsubsection{}
\label{sec:Gamma isom}
Consider the following diagram:
\begin{equation}
\label{eqn:diagram3}
\xymatrix{
H^{d-1}(\GgBT)
\ar[d]_\cong^\alpha
&
\displaystyle\varinjlim_L
H^{d-1}(\GgBT, |L|)
\ar[l]^\beta
\ar[d]^\alpha
\\
H^{d-1}_\Gamma(\gBTdot)
&
\displaystyle\varinjlim_L
H^{d-1}
(\GgBTEG, \wt{|L|} )
\ar[l]^{\beta\phantom{;lkj;;}}
\\
H^{d-1}_\Gamma(\cgBTdot)
\ar[u]_\alpha^\cong
&
\displaystyle\varinjlim_L
H^{d-1}
(\GgbarBTEG, \wt{|L|} \cup
\Gamma \backslash \partial \overline{\gBTdot} \times |E\Gamma_\bullet|
)
\ar[l]^{\beta\phantom{LLLLLLLLLLLLLL}} 
\ar[d]^\beta 
\ar[u]_\alpha^\cong
\\
&
H^{d-1}
(\GgbarBTEG, 
\Gamma \backslash \partial \overline{\gBTdot} \times |E\Gamma_\bullet|)
\ar[lu]^\beta.
}
\end{equation}
Here $L$ runs over the subsimplicial complexes
as in diagram \eqref{eqn:diagram2},
and $\wt{|L|}$
is the inverse image of $|L|$
by the projection
$\GgBTEG \to \GgBT$.
The maps labeled by 
$\alpha$ are (induced by) pullbacks.
The maps $\beta$ are (induced by)
the forget support maps.
The diagram is commutative.

The second vertical arrow on the right hand column
is an isomorphism by the excision property of cohomology.
\begin{lem}
\label{lem:5.5}
The maps on the left column of the diagram \eqref{eqn:diagram3}
are isomorphisms.
\end{lem}
\begin{proof}
As $\BT$ is contractible and $\Gamma$ satisfies (3)
of Section~\ref{sec:def Gamma}
(hence the stabilizer group of a simplex is a finite group
as discussed there),
the group $H^{d-1}(\Gamma \backslash \BT)$
is isomorphic to $H^{d-1}(\Gamma)$.
As $\BT\times |E\Gamma_\bullet|$ and
$\BTbar\times |E\Gamma_\bullet|$ are 
contractible and $\Gamma$ acts freely,
the groups 
$H^{d-1}(\Gamma\backslash \BT\times |E\Gamma_\bullet|)$
and 
$H^{d-1}(\Gamma\backslash \BTbar\times |E\Gamma_\bullet|)$
are also isomorphic to 
 $H^{d-1}(\Gamma)$.
(These statements can be proved using spectral sequences,
which are compatible with pullbacks.)
This implies that the left vertical arrows are isomorphisms.
 \end{proof}

\subsubsection{}
Consider the following diagram.
\begin{equation}
\label{eqn:diagram4}
\xymatrix{
H^{d-1}(\GgBTEG)^*
\ar[r]
&
H^{d-1}(\GgBTEG, \Gamma \backslash \partial \cgBTdot
\times |E\Gamma_\bullet|)^*
\\
H_{d-1}(\GgBTEG)
\ar[u]^\cong
\ar[r]
&
H_{d-1}(\GgBTEG, \Gamma \backslash \partial \cgBTdot
\times |E\Gamma_\bullet|)
\ar[u]
}
\end{equation}
Each of the two vertical arrows in the square 
is the canonical map of the form $A \to A^{**}$.
The lower horizontal arrow is the pushforward map
of  homology.
The top horizontal arrow is the twice dual of the lower horizontal 
arrow, and is the dual of the forget support map of 
cohomology.
\begin{lem}
\label{lem:5.6}
The left vertical map in the diagram \eqref{eqn:diagram4}
is an isomorphism.
\end{lem}
\begin{proof}
As was remarked in the proof of 
Lemma~\ref{lem:5.5}, 
the group $H^{d-1}(\GgBTEG)$
is isomorphic to $H^{d-1}(\Gamma)$.
From property (5) in Section~\ref{sec:def Gamma}, we know that 
$H^{d-1}(\Gamma)$ is a finite dimensional $\Q$-vector 
space. 
This implies the claim.
 \end{proof}
\begin{cor}
\label{cor:5.7}
The map (6) in diagram \eqref{eqn:diagram1}
is an isomorphism.
\end{cor}
\begin{proof}
As remarked in the proof of the previous lemma,
$H^{d-1}(\GgBTEG)$ is finite dimensional.
Using the isomorphisms (8) (10)
and the isomorphisms in diagrams \eqref{eqn:diagram3} and
\eqref{eqn:diagram3}, 
we see that $H_{d-1}(\GBT)^{**}$ 
is also finite dimensional.
The claim follows from this.
 \end{proof}

\subsubsection{}
We define the map (2) in diagram 
\eqref{eqn:diagram0} 
to be the composite
\[
\begin{array}{l}
H_{d-1}(\GgBTEG)
\xto{\cong}
H^{d-1}(\GgBTEG)^*
\xleftarrow{\cong}
H^{d-1}(\GgBT)^*
\\
\xleftarrow{\cong}
H^{d-1}(\GBT)^*
\xleftarrow{\cong}
H_{d-1}(\GBT)
\end{array}
\]
where the maps are
the left vertical arrow in the diagram \eqref{eqn:diagram4},
the dual of the composite of 
the left vertical arrows in \eqref{eqn:diagram3},
the dual of (10),
and 
the dual of the composite (8)(6).

We define the map (3) in diagram 
\eqref{eqn:diagram0} to be
the composite 
\[
\begin{array}{l}
H_{d-1}(\GgbarBTEG, \Gamma \backslash \partial \cgBTdot
\times |E\Gamma_\bullet|)
\to
H^{d-1}(\GgbarBTEG, \Gamma \backslash \partial \cgBTdot
\times |E\Gamma_\bullet|)^*
\\
\to
\varinjlim_L
H^{d-1}(\GgBT, |L|)
\to
H_c^{d-1}(\GBT)^*
\to
H_{d-1}^\BM(\GBT)
\end{array}
\]
where the maps are 
the right vertical arrow 
in the diagram \eqref{eqn:diagram4},
the dual of the vertical arrows 
in the diagram \eqref{eqn:diagram3},
the dual of the composite $(11)^{-1}(12)$
and the inverse of (7).

The diagram \eqref{eqn:diagram0} is then commutative
since each of the diagrams 
\eqref{eqn:diagram1},
\eqref{eqn:diagram2},
\eqref{eqn:diagram3},
\eqref{eqn:diagram4}
is commutative.

\subsection{}
Given a basis $v_1,\dots, v_d$ of $V$, 
we may regard $A_\bullet$ as a subsimplicial complex of $\cBT_\bullet$
using the map $\iota_{v_1,\dots, v_d}$.
This simplicial complex is denoted by
$A_{v, \bullet}=A_{v_1,\dots, v_d, \bullet}$.
Let $\overline{|A_{v,\bullet}|}$ 
denote the closure of $|A_{v, \bullet}|$ 
in $\BTbar$ and set
$\partial\overline{|A_{v,\bullet}|}
=\overline{|A_{v,\bullet}|} \setminus |A_{v, \bullet}|$.

Let $\Delta_{d-1}'$ denote the interior of $\Delta_{d-1}$.
Let $\varphi: \Delta_{d-1}' 
\xto{\cong} 
\R^{d}/\R(1,\dots,1)$  
be the homeomorphism given by 
$(t_0,\dots, t_{d-1}) \mapsto (1/t_0,\dots, 1/t_{d-1})$.
Let $n>0$ be an integer.  
We set $\wt{K_n}=\prod_{i=0}^d[0,n] \subseteq \R^d$,
and let $K_n$ denote the image of 
$\wt{K_n}$ in $\R^d/\R(1,\dots, 1)$.
Recall (Section~\ref{sec:def apartment}) 
that the simplices
of $A_\bullet$ are defined using the set of vertices
$A_0=\Z^{\oplus d}/ \Z(1,\dots,1)$.
We regard $A_0 \subset \R^{\oplus d}/ \R(1,\dots,1)$
using the natural inclusion $\Z \subset \R$.
The set of those simplices of $A_\bullet$ 
whose support is contained in the complement
of the interior of
$K_n$ naturally forms a subsimplicial complex of 
$A_\bullet$.  We call this simplicial complex
$K_{n,\bullet}^c$.  It is easy to see that
$\cap_n K_{n,\bullet}^c =\emptyset$.

Consider the following map
\begin{eqnarray}
\label{eqn:sub diagram}\\
\nonumber 
H_{d-1}(\overline{|A_{v, \bullet}|}, \partial\overline{|A_{v,\bullet}|})
\rightarrow
(H^{d-1}(\overline{|A_{v, \bullet}|}, \partial\overline{|A_{v, \bullet}|}))^* 
\rightarrow
(\displaystyle\varinjlim_n 
H^{d-1}(\overline{|A_{v,\bullet}|}, 
\partial\overline{|A_{v, \bullet}|} \cup |K_{n,\bullet}^c|))^*      
\\
\nonumber
\xleftarrow{\cong}
(\displaystyle\varinjlim_n H^{d-1}(|A_{v, \bullet}|, |K_{n,\bullet}^c|))^*  
\xleftarrow{\cong}
(\displaystyle\varinjlim_n H^{d-1}(A_{v, \bullet}, K_{n,\bullet}^c))^*  
\xleftarrow{\cong}
H_c^{d-1}(A_{v,\bullet})^* 
\xleftarrow{\cong}
H_{d-1}^\BM(A_{v,\bullet})    
\end{eqnarray}
where the limit is over the nonnegative integers
in each case.
The first map is the canonical map $A \to A^{**}$
for $A=H_{d-1}(\overline{|A_{v, \bullet}|}, 
\partial\overline{|A_{v,\bullet}|})$.
The second map is the dual of the 
limit of the pullback map at each stage.
The third map is the dual of the limit of the
excision isomorphism of cohomology.
The fourth map is the dual of the limit of
the isomorphisms
between cellular cohomology and 
singular cohomology (see Section~\ref{sec:cellular}).
The fifth map is the map obtained from the 
definitions in Section~\ref{sec:def homology}.
It is an isomorphism since 
$\cap_n K_{n,\bullet}^c =\emptyset$.
The sixth map is the duality isomorphism in the 
universal coefficient theorem (see Section~\ref{univ_coeff}).

Note that the image of the continuous map 
$s(v_1, \dots, v_d)$
of Section~\ref{sec:4 and 5} is contained in 
$\overline{|A_{v, \bullet}|}$
and defines a class $[s(v_1,\dots, v_d)]$ 
in
$H_{d-1}(
\overline{|A_{v, \bullet}|} 
,
\partial
\overline{|A_{v, \bullet}|} 
)$.

\begin{lem}
\label{lem:geom apartment}
The following two elements in
$H^\BM_{d-1}
(A_\bullet)$
coincide:
\begin{enumerate}
\item 
The image of the class of $s(v_1,\dots, v_d)$
by 
\[
H_{d-1}(\overline{|A_{v,\bullet}|}, \partial\overline{|A_{v,\bullet}|})
\to
H_{d-1}^\BM(A_{v, \bullet})
\xto{\iota_{v_1,\dots,v_d}^{-1}}
H_{d-1}^\BM(A_\bullet),
\]
where the first map is the map in the diagram \eqref{eqn:sub diagram}.
\item
The class of $\beta$ of Section~\ref{sec:fundamental class}
in $H_{d-1}^\BM(A_\bullet)$.
\end{enumerate}
\end{lem}

\begin{proof}
Let us describe the image of the class 
$[s(v_1,\dots, v_d)]$ in 
$[\varinjlim_n
H^{d-1}
(A_\bullet, K_{n,\bullet}^c)]^*$.
Take an element
$h \in 
\varinjlim_n
H^{d-1}
(A_\bullet, K_{n,\bullet}^c).$
We may suppose it is represented by an element $h_m$ of 
$H^{d-1}
(A_\bullet, K_{m,\bullet}^c)$
for some $m$.

Consider the map
\[
H_{d-1}(\overline{|A_\bullet|}, \partial \overline{|A_\bullet|})
\to
H_{d-1}(\overline{|A_\bullet|}, 
\partial \overline{|A_\bullet|} \cup |K_{m, \bullet}^c|)
\xleftarrow{\cong}
H_{d-1}(|A_\bullet|, |K_{m,\bullet}^c|)
\xleftarrow{\cong}
H_{d-1}(A_\bullet, K_{m,\bullet}^c).
\]
Let $t_m$ denote the image of  
$[s(v_1,\dots, v_d)]$ via this map.
Then the pairing of $[s(v_1,\dots, v_d)]$
with the element $h$ is the pairing 
of $h_m$ and $t_m$ under  the canonical pairing
\[
H_{d-1}(A_\bullet, K_{m,\bullet}^c)
\times 
H^{d-1}(A_\bullet, K_{m,\bullet}^c)
\to \Q
\]
between homology and cohomology.

Let us compute $t_m$.
Given $s=(s_0,\dots, s_{d-1}) \in K_m$, 
take a representative 
$\wt{s}=(\wt{s_0},\dots, \wt{s_{d-1}}) 
\in \R^d$ such that
$-m \le s_i \le 0$ 
$(0\le i \le d-1), 
\min_i s_i=-m$. 
We define a map
$g_m:K_m \to \Delta_{d-1}$ 
by 
$s \mapsto (\wt{s_0}/(\wt{s_0}+\cdots+\wt{s_{d-1}}), 
\dots, \wt{s_{d-1}}/(\wt{s_0}+\cdots+\wt{s_{d-1}}))$.
It is well-defined and is a homeomorphism.
Let $f_m: \Delta_{d-1} \to \Delta_{d-1}$ 
be the composite 
$\Delta_{d-1} \xto{g_n^{-1}} 
K_m \subset \R^d/\R(1,\dots,1) 
\xleftarrow{\cong} \Delta_{d-1}' \subset \Delta_{d-1}$.
From the following lemma, it follows that the class of 
$s(v_1,\dots, v_d)$
equals the class of 
$s(v_1,\dots,v_d)\circ f_m$
in $H_{d-1}(\overline{|A_\bullet|}, 
\partial\overline{|A_\bullet|} \cup |K_{m,\bullet}^c|).$
Note that $s(v_1,\dots, v_d) \circ f_m$ 
defines a class in $H_{d-1}(|A_\bullet|,|K_{m,\bullet}^c|)$,
hence this is the image of $[s(v_1,\dots, s_d)]$.

It is easy to check that 
the class of 
$s(v_1,\dots,v_d)\circ f_m$ is then represented by the chain
$(\gamma_\nu)\in \prod_{\nu \in A_{d-1}'}\Z$
where $\gamma_\nu=\beta_\nu$ for $\nu \in K_m'$
($\beta_\nu$ was defined in Section~\ref{sec:fundamental class}),
and $\gamma_\nu=0$ for $\nu \notin K_m'$.
 \end{proof}

\begin{lem}
Let $X$ be a topological space and $Y$ a subspace.
Let $n \ge 1$.
Let $\alpha:\Delta_r \to X$ be a continuous map such that $\alpha(\overline{(\Delta_r\setminus \mathrm{Im} f_n)})\subset Y$.
Then $\alpha$ and $\alpha\circ f_n: \Delta_r \to X$ both define the same element in $H_r(X,Y;\Z)$.
\end{lem}
\begin{proof}
Omitted.
 \end{proof}

\subsection{Proof of Theorem \ref{lem:apartment}}
\begin{proof}
By Lemma~\ref{lem:geom apartment},
it suffices to show that 
\begin{enumerate}
\item
the image of the class of 
$[A_{v_1,\dots, v_d;g_0,\dots, g_d}]$
in $H^\BM_{d-1}(\Gamma \backslash \cBT_\bullet)$,
and
\item
the image of the class of $s(v_1,\dots, v_d)$
via the composite map
\[
H_{d-1}(\overline{|A_{v,\bullet}|}, 
\partial \overline{|A_{v,\bullet}|})
\to
H_{d-1}^\BM(A_{v,\bullet})
\to
H_{d-1}^\BM
(\Gamma \backslash \cBT_\bullet),
\]
where the first map is the 
map in the diagram \eqref{eqn:sub diagram},
\end{enumerate}
coincide.

The argument is similar to that in 
the proof of Lemma~\ref{lem:geom apartment}.
We compare the two classes in 
\[
[\varinjlim_L H^{d-1}(\GgBT,|L|)]^*
\]
which appeared in the definition of the map (3) 
in diagram \eqref{eqn:diagram0}.
Let $h$ be an element in 
$H^{d-1}(\GgBT, |L|)$.
We need to compute the images of the two classes in
\[
H_{d-1}(\GgBT, |L|)
\]
Since the class (1) is represented by 
the chain
$s(v_1, \dots, v_d) \times 
[g_0, \dots, g_{d-1}]$,
using the argument as in the proof of Lemma~\ref{lem:geom apartment},
we see that it is represented by the chain
$s(v_1, \dots, v_d)\circ f_n$
for sufficiently large $n$ 
in $H_{d-1}(\GgBT, |L|)$.
Again, as we have seen in the proof of Lemma~\ref{lem:geom apartment},
this is nothing but the class of (2).
 \end{proof}

\section{The homology of an arithmetic quotient}
\label{section8}
In this section, we compute the homology groups and 
the Borel-Moore homology groups 
of some arithmetic quotients of
the Bruhat-Tits building and relate them to the space 
of automorphic forms.
The aim of this section is to prove 
Proposition \ref{7_prop1} below.

\subsection{Identification of homology groups and the space of automorphic forms}
\label{subsec:X_K}
\ 

For an open compact subgroup
$\bK \subset \GL_d(\A^\infty)$,
we let 
$\wt{X}_{\GL_d,\bK,\bullet}$ denote the disjoint union
$\wt{X}_{\GL_d,\bK,\bullet}=(\GL_d(\A^\infty)/\bK) 
\times \cBT_{\bullet}$
of copies of the Bruhat-Tits building $\cBT_{\bullet}$ 
indexed by $\GL_d(\A^\infty)/\bK$.
We often omit the subscript $\GL_d$ on
$\wt{X}_{\GL_d,\bK,\bullet}$ when there is no
fear of confusion. The group $\GL_d(\A)$ acts on
the simplicial complex $\wt{X}_{\bK,\bullet}$ from the left.
We study the quotient $\GL_d(F) \backslash \wt{X}_{\bK,\bullet}$ of 
$\wt{X}_{\bK,\bullet}$ by the subgroup
$\GL_d(F) \subset \GL_d(\A)$.

For $0\le i \le d-1$, we let $X_{\bK,i}= X_{\GL_d,\bK,i}$ 
denote the quotient $X_{\bK,i} = \GL_d(F) \backslash 
\wt{X}_{\GL_d,\bK,i}$.
We set $J_\bK =\GL_d(F) \backslash \GL_d(\A^\infty)/\bK$.
For each $j \in J_\bK$, we choose an element 
$g_j \in \GL_d(\A^\infty)$ in the double coset $j$
and set $\Gamma_j = \GL_d(F) \cap g_j \bK g_j^{-1}$.
Then the set $X_{\bK,i}$ is isomorphic to the disjoint union
$\coprod_j \Gamma_j \backslash \cBT_i$. For each $j$,
the group $\Gamma_j \subset \GL_d(F)$ is an arithmetic 
group as defined in Section~\ref{sec:71}.  It follows that the tuple
$X_{\bK,\bullet}=(X_{\bK,i})_{0\le i \le d-1}$ forms a simplicial complex 
which is isomorphic to the disjoint union 
$\coprod_{j\in J_\bK} \Gamma_j \backslash \cBT_\bullet$.

Since the simplicial complex $\wt{X}_{\GL_d,\bK,\bullet}$ is locally
finite, it follows that the simplicial complex 
$X_{\bK,\bullet}$ is locally finite.
Hence for an abelian group $M$, we may consider the
cohomology groups with compact support $H_c^{*}(X_{\bK,\bullet},M)$
(\resp the Borel-Moore homology groups $H_*^\mathrm{BM}(X_{\bK,\bullet},M)$) 
of the simplicial complex $X_{\bK,\bullet}$.
%

Since the simplicial complex $X_{\bK,\bullet}$
has no $i$-simplex for $i \ge d$ as was remarked in
Section~\ref{sec:dimension}, 
it follows that
the map
$$
H_{d-1}(X_{\bK,\bullet},M) \to 
H_{d-1}^\mathrm{BM}(X_{\bK,\bullet},M)
$$
is injective for any abelian group $M$.
We regard $H_{d-1}(X_{\bK,\bullet},M)$ as a subgroup of
$H_{d-1}^\mathrm{BM}(X_{\bK,\bullet},M)$.

\subsubsection{}
Let $\St_d$ denote the Steinberg representation
as defined, for example, in \cite[p.193]{Laumon1}.
It is defined with coefficients in $\C$,
but it can also be defined with coefficients in $\Q$
in a similar manner.  We let $\St_d$ denote the
corresponding representation.

\begin{lem}\label{lem:Steinberg}
For a $\Q$-vector space $M$, there is a canonical,
$\GL_d(F_\infty)$-equivariant isomorphism 
between the module of $M$-valued 
harmonic $(d-1)$-cochains and the module
$\Hom_{\Q}(\St_d,M)$.
\end{lem}

\begin{proof}
By definition, the module of $M$-valued 
harmonic $(d-1)$-cochains is identified with
$\Hom(H^{d-1}_c(\cBT_\bullet,\Q),M)$.
It is shown in~\cite[6.2,6.4]{Borel} that
$\St_d$ (with $\C$-coefficient) 
is canonically isomorphic to
$H^{d-1}_c(\cBT_\bullet,\C)$ as a 
representation of $\GL_d(F_\infty)$.
One can check that 
this map is defined over $\Q$. 
This proves the claim.
\end{proof}

\subsubsection{}
\label{sec:6.1.2}
We let $\cBT_{j,*}$ 
denote the quotient
$\wt{\cBT}_j/F_{\infty}^{\times}$.
This set 
is identified with the set of 
pairs $(\sigma, v)$ with 
$\sigma \in \cBT_j$
and $v \in \cBT_0$ a vertex of $\sigma$, 
which we call 
a pointed $j$-simplex.
Here the element 
$(L_i)_{i\in \Z} \mod K^\times$
of $\wt{\cBT_j}/K^\times$
corresponds to the pair
$((L_i)_{i\in \Z}, L_0)$
via this identification.

We identify the set
$\wt{\cBT}_{0}$ with the coset
$\GL_d(K)/\GL_d(\cO)$
by associating to an element $g \in
\GL_d(K)/\GL_d(\cO)$ the lattice
$\cO_{V}g^{-1}$.
Let 
$\cI=\{(a_{ij})\in \GL_d(\cO) \,|\,
a_{ij}\,\mathrm{mod}\, \varpi =0 \ \text{if}\ i>j\}$
be the Iwahori subgroup.
Similarly, we identify the set $\wt{\cBT}_{d-1}$
with the coset $\GL_d(K)/\cI$
by associating
to an element $g\in \GL_d(K)/\cI$
the chain of lattices $(L_i)_{i \in \Z}$ characterized by
$L_i=\cO_{V}\Pi_ig^{-1}$ for $i=0,\dots,d$.
Here, for $i=0,\dots,d$,
we let $\Pi_i$ denote the diagonal $d\times d$ matrix
$\Pi_i=\diag(\varpi,\ldots,\varpi,1,\ldots,1)$
with $\varpi$ appearing $i$
times and $1$ appearing $d-i$ times.

Let $M$ be a $\Q$-vector space.
Let $\cC^\bK(M)$ denote the 
($\Q$-vector)  space of locally constant 
$M$-valued functions on
$\GL_d(F)\backslash \GL_d(\A)/(\bK \times F_\infty^\times)$.
Let $\cC_c^\bK(M) \subset C^\bK(M)$
denote the subspace of compactly supported functions.
\begin{lem}\label{lem:Steinberg8}
\begin{enumerate}
\item There is a canonical isomorphism
$$
H_{d-1}^\mathrm{BM}(X_{\bK,\bullet},M) 
\cong \Hom_{\GL_d(F_\infty)} (\St_d, \cC^\bK(M)),
$$
where $\cC^\bK(M)$ 
denotes the space of locally constant $M$-valued
functions on $\GL_d(F)\backslash \GL_d(\A)/(\bK\times F_\infty^\times)$.
\item 
Let $v \in \St_d^\cI$ be a non-zero 
Iwahori-spherical vector. 
Then the image of the evaluation map
$$
\begin{array}{l}
\Hom_{\GL_d(F_\infty)}(\St_d, \cC^\bK(M)) \\
\to \Map(\GL_d(F)\backslash \GL_d(\A)/(\bK \times 
F_\infty^\times \cI),M)
\end{array}
$$
at $v$ is identified with the image of
the map
$$
\begin{array}{rl}
H_{d-1}^\mathrm{BM}(X_{\bK,\bullet},M) 
 & \to \Map(\GL_d(F)\backslash 
(\GL_d(\A^\infty)/\bK \times \cBT_{d-1,*}),M) \\
& \cong \Map(\GL_d(F)\backslash \GL_d(\A)/(\bK \times 
F_\infty^\times \cI),M).
\end{array}
$$
\end{enumerate}
\end{lem}
\begin{proof}
For a $\C$-vector space $M$,
(1) is proved in \cite[Section 5.2.3]{KY:Zeta elements},
and (2) is \cite[Corollary 5.7]{KY:Zeta elements}.
The proofs and the argument in loc. cit. 
work for a $\Q$-vector space $M$ as well.
\end{proof}

\begin{cor}
\label{cor:Steinberg8}
Under the isomorphism in (1), the subspace
$$
H_{d-1}(X_{\bK,\bullet},M) \subset 
H_{d-1}^\mathrm{BM}(X_{\bK,\bullet},M)
$$
corresponds to the subspace 
$$
\Hom_{\GL_d(F_\infty)}(\St_d, \cC_c^\bK(M))
\subset 
\Hom_{\GL_d(F_\infty)}(\St_d, \cC^\bK(M)).
$$
\end{cor}
\begin{proof}
This follows from Lemma~\ref{lem:Steinberg8} (2)
and the definition of the homology group
$H_{d-1}(X_{\bK, \bullet}, M)$.
\end{proof}

\subsection{Pull-back maps for homology groups}
\ 

Let $\bK,\bK' \subset
\GL_d(\A^\infty)$ be open compact subgroups 
with $\bK' \subset \bK$. We denote by 
$f_{\bK',\bK}$ the natural projection map
$X_{\bK',i} \to X_{\bK,i}$.
Since $\bK'$ is a subgroup of
$\bK$ of finite index, it follows that 
for any $i$ with $0 \le i \le d-1$ 
and for any $i$-simplex $\sigma \in X_{\bK,i}$, 
the inverse image of $\sigma$ under the
map $f_{\bK',\bK}$ is a finite set.
Let $i$ be an integer with $0 \le i \le d-1$
and let $\sigma' \in X_{\bK',i}$. 
Let $\sigma$ denote the image of $\sigma'$
under the map $f_{\bK',\bK}$.
Let us choose an $i$-simplex $\wt{\sigma}'$ of
$\wt{X}_{\bK',\bullet}$ which is sent to $\sigma'$
under the projection map $\wt{X}_{\bK',\bullet}
\to X_{\bK',\bullet}$.
Let $\wt{\sigma}$ denote the image of $\wt{\sigma}'$
under the map $\wt{X}_{\bK',i} \to \wt{X}_{\bK,i}$.
We let
$$
\Gamma_{\wt{\sigma}'} = \{ \gamma \in \GL_d(F)\ |\ 
\gamma \wt{\sigma}' =\wt{\sigma}' \}
$$
and
$$
\Gamma_{\wt{\sigma}} = \{ \gamma \in \GL_d(F)\ |\ 
\gamma \wt{\sigma} =\wt{\sigma} \}
$$
denote the isotropy group of $\wt{\sigma}'$ and
$\wt{\sigma}$, respectively.

The following lemma can be checked easily.
\begin{lem}
Let the notation be as above.
\begin{enumerate}
\item The group $\Gamma_{\wt{\sigma}}$ is a finite group and
the group $\Gamma_{\wt{\sigma}'}$ is a subgroup of $\Gamma_{\wt{\sigma}}$.
\item The isomorphism class of the group $\Gamma_{\wt{\sigma}'}$
(\resp $\Gamma_{\wt{\sigma}}$) depends only on $\sigma'$
(\resp $\sigma$) and does not depends on the choice of
$\wt{\sigma}'$. 
\end{enumerate}
 
\end{lem}

The lemma above shows in particular that the index
$[\Gamma_{\wt{\sigma}} : \Gamma_{\wt{\sigma}'}]$ is finite and
depends only on $\sigma'$ and $f_{\bK',\bK}$. We denote
this index by $e_{\bK',\bK}(\sigma')$ and call it the
ramification index of $f_{\bK',\bK}$ at $\sigma'$.

Let $M$ be an abelian group.
Let $i$ be an integer with $0 \le i \le d$.
We set $X'_{\bK,i}= \coprod_{\sigma \in X_{\bK,i}} O(\sigma)$.
The map $f_{\bK',\bK} : X_{\bK',\bullet} \to X_{\bK,\bullet}$
induces a map $X'_{\bK',i} \to X'_{\bK,i}$ which we denote
also by $f_{\bK'.\bK}$.
Let $m = (m_{\nu})_{\nu \in X'_{\bK,i}}$ be an element
of the $\{\pm 1\}$-module $\prod_{\nu \in X'_{\bK,i}} M$.
We define the element $f^*_{\bK',\bK}(m)$
in $\prod_{\nu \in X'_{\bK',i}} M$ to be
$$
f^*_{\bK',\bK}(m) = (m'_{\nu'})_{\nu' \in X'_{\bK',i}}
$$
where for $\nu' \in O(\sigma') \subset X'_{\bK',i}$,
the element $m'_{\nu'} \in M$ is given by
$m'_{\nu'} = e_{\bK',\bK}(\sigma') m_{f_{\bK',\bK}(\nu')}$.
The following lemma can be checked easily.
\begin{lem}
Let the notation be as above.
\begin{enumerate}
\item The map $f^*_{\bK',\bK} : \prod_{\nu \in X'_{\bK,i}} M
\to \prod_{\nu' \in X'_{\bK',i}} M$ is a homomorphism of
$\{\pm 1\}$-modules.
\item The map $f^*_{\bK',\bK} : \prod_{\nu \in X'_{\bK,i}} M
\to \prod_{\nu' \in X'_{\bK',i}} M$ sends an element in 
the subgroup $\bigoplus_{\nu \in X'_{\bK,i}} M 
\subset \prod_{\nu \in X'_{\bK,i}} M$ to an element
in  $\bigoplus_{\nu \in X'_{\bK',i}} M$.
\item For $1 \le i \le d-1$, the diagrams
$$
\begin{CD}
\prod_{\nu \in X'_{\bK,i}} M @>{\wt{\partial}_{i,\prod}}>> 
\prod_{\nu \in X'_{\bK,i-1}} M \\
@V{f_{\bK',\bK}^*}VV @V{f_{\bK',\bK}^*}VV \\
\prod_{\nu' \in X'_{\bK',i}} M @>{\wt{\partial}_{i,\prod}}>> 
\prod_{\nu' \in X'_{\bK',i-1}} M
\end{CD}
$$
and
$$
\begin{CD}
\bigoplus_{\nu \in X'_{\bK,i}} M @>{\wt{\partial}_{i,\oplus}}>> 
\bigoplus_{\nu \in X'_{\bK,i-1}} M \\
@V{f_{\bK',\bK}^*}VV @V{f_{\bK',\bK}^*}VV \\
\bigoplus_{\nu' \in X'_{\bK',i}} M @>{\wt{\partial}_{i,\oplus}}>> 
\bigoplus_{\nu' \in X'_{\bK',i-1}} M
\end{CD}
$$
are commutative.
\end{enumerate}
 
\end{lem}

The lemma above shows that the map $f^*_{\bK',\bK}$
induces homomorphisms $H_*(X_{\bK,\bullet},M)
\to H_*(X_{\bK',\bullet},M)$ and
$H^\mathrm{BM}_*(X_{\bK,\bullet},M)
\to H^\mathrm{BM}_*(X_{\bK',\bullet},M)$
of abelian groups. We denote these homomorphisms
also by $f_{\bK',\bK}^*$.
We remark here that in \cite[p.561, 5.3.3]{KY:Zeta elements},
we implicitly use these pullback maps for the Borel-Moore
homology.

The proof of the following lemma is straightforward
and is left to the reader.
\begin{lem}
Let the notation be as above.
\begin{enumerate}
\item 
Suppose that $\bK'$ is a normal subgroup of $\bK$.
Then the homomorphism $f^*_{\bK',\bK}$
induces an isomorphism $H^\mathrm{BM}_*(X_{\bK,\bullet},M) 
\cong H^\mathrm{BM}_*(X_{\bK',\bullet},M)^{\bK/\bK'}$
and a similar statement holds for 
$H_*$.
\item
Let $M$ be a $\Q$-vector space. Then the diagrams
$$
\begin{CD}
H^\mathrm{BM}_{d-1} (X_{\bK,\bullet}, M) @>{\cong}>>
Hom_{\GL_d(F_\infty)}(\St_d, \cC^\bK(M)) \\
@V{f^*_{\bK',\bK}}VV @VVV \\
H^\mathrm{BM}_{d-1} (X_{\bK',\bullet}, M) @>{\cong}>>
Hom_{\GL_d(F_\infty)}(\St_d, \cC^{\bK'}(M))
\end{CD}
$$
and
$$
\begin{CD}
H_{d-1} (X_{\bK,\bullet}, M) @>{\cong}>>
Hom_{\GL_d(F_\infty)}(\St_d, \cC_c^\bK(M)) \\
@V{f^*_{\bK',\bK}}VV @VVV \\
H_{d-1} (X_{\bK',\bullet}, M) @>{\cong}>>
Hom_{\GL_d(F_\infty)}(\St_d, \cC_c^{\bK'}(M))
\end{CD}
$$
are commutative.
Here the horizontal arrows are the isomorphisms
given in Lemma~\ref{lem:Steinberg8} 
and Corollary~\ref{cor:Steinberg8}, 
and the right vertical
arrows are the map induced by the quotient map
$\GL_d(F) \backslash \GL_d(\A) /(\bK' \times F_\infty^\times)
\to \GL_d(F) \backslash \GL_d(\A) /(\bK\times F_\infty^\times)$.
\end{enumerate}
 
\end{lem}

\subsection{The action of $\GL_d(\A^\infty)$ and admissibility}
For $g \in \GL_d(\A^\infty)$, we let 
$\wt{\xi}_g : \wt{X}_{\bK,\bullet}
\xto{\cong} \wt{X}_{g^{-1} \bK g,\bullet}$
denote the isomorphism of simplicial complexes
induced by the isomorphism $\GL_d(\A^\infty)/\bK \xto{\cong}
\GL_d(\A^\infty)/g^{-1}\bK g$ which sends 
a coset $h \bK$ to the coset $hg\cdot g^{-1}
\bK g$ and by the identity on $\cBT_{\bullet}$.
The isomorphism $\wt{\xi}_g$ induces an isomorphism 
$\xi_{g} :X_{\bK,\bullet}
\xto{\cong} X_{g^{-1} \bK g,\bullet}$ of
simplicial complexes. 
%
For two elements
$g, g' \in \GL_d(\A^\infty)$, we have
$\xi_{gg'}=\xi_{g'}\circ \xi_g$.

For an abelian group $M$, we let
$H_*(X_{\lim,\bullet},M)=H_*(X_{\GL_d,\lim,\bullet},M)$ and
$H_*^\mathrm{BM}(X_{\lim,\bullet},M)=H_*^\mathrm{BM}(X_{\GL_d,\lim,\bullet},M)$
denote the
inductive limits
$\varinjlim_{\bK}H_*(X_{\bK,\bullet},M)$ and
$\varinjlim_{\bK}H_*^\mathrm{BM}(X_{\bK,\bullet},M)$,
respectively.
Here the transition maps in the inductive limits are 
given by $f^*_{\bK',\bK}$.
The isomorphisms
$\xi_g$ for $g\in \GL_d(\A^\infty)$ gives
rise to a smooth action of the group $\GL_d(\A^\infty)$
on these inductive limits.
If $M$ is a torsion free abelian group, then
for each compact open subgroup $\bK \subset \GL_d(\A^\infty)$,
the homomorphism
$H_*(X_{\bK,\bullet},M) \to 
H_*(X_{\lim,\bullet},M)$
is injective and its image is equal to the
$\bK$-invariant part $H_*(X_{\lim,\bullet},M)^{\bK}$
of $H_*(X_{\lim,\bullet},M)$.
Similar statement holds for $H_*^\mathrm{BM}$.

\begin{lem}\label{7_lem1}
\begin{enumerate}
\item
For any open compact subgroup 
$\bK \subset \GL_d(\A^\infty)$,
both $H_{d-1}^\mathrm{BM}(X_{\bK,\bullet},\Q)$ and
$H_{d-1}(X_{\bK,\bullet},\Q)$ are finite dimensional.
\item
The inductive limits
$H_{d-1}(X_{\lim,\bullet},\Q)$ and
$H_{d-1}^\mathrm{BM}(X_{\lim,\bullet},\Q)$ are
admissible $\GL_d(\A^\infty)$-modules.
\end{enumerate}
\end{lem}

\begin{proof}
It follows from Lemma~\ref{lem:Steinberg8} that
the $\Q$-vector space
$H_{d-1}^\mathrm{BM}(X_{\bK,\bullet},\Q)$ is isomorphic
to $\Hom_{\GL_d(F_\infty)}(\St_d,\cC^\bK(\Q))$.
Let $\cH_\infty$ denote the convolution algebra of
locally constant, compactly supported $\Q$-valued functions 
on $\GL_d(F_\infty)$, with respect to a Haar measure on
$\GL_d(F_\infty)$ such that the volume of $\GL_d(\cO_\infty)$
is a rational number.
We regard $\St_d$ as a left $\cH_\infty$-module.
Let $K_\infty \subset \GL_d(F_\infty)$ be a compact open
subgroup such that the $K_\infty$-invariant part
$\St_d^{K_\infty}$ is non-zero.
Let us fix a non-zero vector $v \in \St_d^{K_\infty}$ and let
$J \subset\cH_\infty$ denote the set of elements $f \in \cH_\infty$
such that $fv =0$.
Then $J$ is an admissible left ideal of $\cH_\infty$
in the sense of \cite[5.5, p.199]{BJ}.
The map $\Hom_{\GL_d(F_\infty)}(\St_d,\cC^\bK(\Q))
\to \cC^\bK(\Q)$ which sends $\varphi: \St_d \to \cC^\bK(\Q)$
to $\varphi(v)$ gives an isomorphism
from the space $\Hom_{\GL_d(F_\infty)}(\St_d,\cC^\bK(\Q))$
to the space of $\Q$-valued functions
on $\GL_d(F) \backslash \GL_d(\A)$ which is right invariant
under $\bK \times K_\infty$ and is annihilated by $J$.
Hence it follows from \cite[5.6.\ THEOREM, p.199]{BJ} that
the space $\Hom_{\GL_d(F_\infty)}(\St_d,\cC^\bK(\Q))$ is 
finite dimensional. 
Therefore $H_{d-1}^\mathrm{BM}(X_{\bK,\bullet},\Q)$ is finite
dimensional. The space $H_{d-1}(X_{\bK,\bullet},\Q)$ is
finite dimensional since it is a subspace of
$H_{d-1}^\mathrm{BM}(X_{\bK,\bullet},\Q)$. This proves the
claim (1).

For each compact open subgroup $\bK \subset \GL_d(\A^\infty)$,
the maps $H_{d-1}^\mathrm{BM}(X_{\bK,\bullet},\Q)
\to H_{d-1}^\mathrm{BM}(X_{\lim,\bullet},\Q)^\bK$
and $H_{d-1}(X_{\bK,\bullet},\Q)
\to H_{d-1}(X_{\lim,\bullet},\Q)^\bK$
are isomorphisms.
Hence the claim (2) follows from the claim (1).
\end{proof}

\subsection{The homology}
\begin{prop} \label{prop:66_3}
As a representation of $\GL_d(\A^\infty)$,
the inductive limit $H_{d-1}(X_{\lim,\bullet},\C)$ is isomorphic
to the direct sum
$$
H_{d-1}(X_{\lim,\bullet},\C) \cong
\bigoplus_\pi \pi^\infty,
$$
where $\pi = \pi^\infty \otimes \pi_\infty$ runs over
(the isomorphism classes of) the irreducible cuspidal
automorphic representations 
of $\GL_d(\A)$ such that $\pi_\infty$ is isomorphic to 
the Steinberg representation of $\GL_d(F_\infty)$.
\end{prop}

\begin{proof}
Let $\cC(\C)$ be the space of locally constant, compactly supported
$\cC$-valued functions on $\GL_d(F) \backslash \GL_d(\A) /F_\infty^\times$.
It follows from Corollary~\ref{cor:Steinberg8} that 
the isomorphism in Lemma~\ref{lem:Steinberg8} (1) induces 
a $\GL_d(\A) = \GL_d(\A^\infty) \times \GL_d(F_\infty)$-equivariant
homomorphism
$$
\iota : H_{d-1}(X_{\lim,\bullet},\C) \otimes_\Q \St_d \to \cC(\C).
$$
We denote by $\cA$ the image of the homomorphism $\iota$.
It follows from Corollary~\ref{cor:Steinberg8} that
the map $H_{d-1}(X_{\lim,\bullet},\C) \otimes_\Q \St_d^\cI
\to \cA^\cI$ is an isomorphism. We prove that
$\cA^\cI$ is isomorphic to the right hand side of the
desired isomorphism.
Since $\cC(\C)$ consists of compactly supported functions,
$\cC(\C)$ can be regarded as a subspace of 
$L^2(\GL_d(F) \backslash \GL_d(\A)/F_\infty^\times)$.
It follows from Lemma~\ref{7_lem1} (2) that
$H_{d-1}(X_{\lim,\bullet},\C) \otimes_\Q \St_d$
is an admissible representation of $\GL_d(\A)$.
Hence $\cA$ is also an admissible representation of $\GL_d(\A)$.
Since $\cA$ is an admissible subrepresentation of 
$L^2(\GL_d(F) \backslash \GL_d(\A)/F_\infty^\times)$,
it follows that $\cA$ is contained in a discrete spectrum of
$L^2(\GL_d(F) \backslash \GL_d(\A)/F_\infty^\times)$
and is a direct sum of irreducible admissible representations.
Let $\pi \subset \cA$ be an irreducible subrepresentation.
It follows from the construction of $\cA$ that the component
$\pi_\infty$ at $\infty$ of $\pi$ is isomorphic to the
Steinberg representation $\St_d$.
It follows from the classification (\cite[p.606, Th\'eor\`eme]{MW}) 
of the discrete 
spectrum of $L^2(\GL_d(F) \backslash \GL_d(\A)/F_\infty^\times)$
that $\pi$ does not belong to the residual spectrum.
Hence $\pi$ is an irreducible cuspidal automorphic representation.
It follows from the multiplicity one theorem 
that $\cA$ is isomorphic to the direct sum of 
the irreducible subrepresentations of $\cA$.
Hence to prove the claim, it suffices to show that
any cuspidal irreducible automorphic subrepresentation 
$\pi$ of $L^2(\GL_d(F) \backslash \GL_d(\A)/F_\infty^\times)$ 
whose component $\pi_\infty$ at $\infty$ is isomorphic to $\St_d$
is contained in $\cA$.
It is essentially proved in \cite[Theorem 1.2.1]{Harder}
(cf. \cite[p.16]{Laumon2}) that the support of a cusp form 
on $\GL_d(\A)$ is compact modulo center.
It follows that $\pi$ is a subspace of $\cC(\C)$. 
Let us write $\pi = \pi^\infty \otimes \pi_\infty$.
Since $\pi_\infty$ is isomorphic to $\St_d$, there
exists, for any vector $v \in \pi^\infty$, a 
$\GL_d(F_\infty)$-equivariant homomorphism $\St_d \to \cC(\C)$
whose image contains $\Q v \otimes \St_d$.
Hence it follows from Corollary~\ref{cor:Steinberg8} that
$\Q v \otimes \St_d$ is contained in $\cA$. Therefore
$\pi$ is a subrepresentation of $\cA$. This proves the claim.
\end{proof}

\begin{remark}
\label{rmk:HMW}
We can prove Proposition~\ref{prop:66_3} 
also by using the argument in \cite{Harder}:
it follows that $H_{d-1}(X_{\lim,\bullet},\C)$ is isomorphic to
the subspace $H$ of $L^2(\GL_d(F) F_\infty^\times \backslash \GL_d(\A))$
spanned by the subrepresentations whose component at $\infty$
is isomorphic to the Steinberg representation.
Then the classification \cite[p.606, Th\'eor\`eme]{MW}
shows that any constituent of $H$ is an irreducible cuspidal automorphic 
representation.
\end{remark}

\section{The Borel-Moore homology of an arithmetic quotient}
\label{sec:BMquot}
The goal of this section is to prove the following theorem.
\begin{thm}\label{7_prop1}
Let $\pi = \pi^\infty \otimes \pi_\infty$ be an irreducible
smooth representation of $\GL_d(\A)$
such that $\pi^\infty$ appears as a subquotient of
$H_{d-1}^\mathrm{BM}(X_{\lim,\bullet},\C)$.
Then there exist an integer $r \ge 1$, 
a partition $d=d_1 + \cdots + d_r$ of $d$,
and irreducible cuspidal automorphic representations $\pi_i$
of $\GL_{d_i}(\A)$ for $i=1,\ldots,d$ which satisfy
the following properties:
\begin{itemize}
\item For each $i$ with $0 \le i \le r$,
the component $\pi_{i,\infty}$ at $\infty$ of $\pi_i$ is 
isomorphic to the
Steinberg representation of $\GL_{d_i}(F_\infty)$.
\item Let us write $\pi_i = \pi_i^\infty \otimes \pi_{i,\infty}$.
Let $P \subset \GL_d$ denote the standard parabolic subgroup
corresponding to the partition $d=d_1 + \cdots + d_r$.
Then $\pi^\infty$ is isomorphic to a subquotient of the unnormalized 
parabolic induction $\Ind_{P(\A^\infty)}^{\GL_d(\A^\infty)}
\pi_1^\infty \otimes \cdots \otimes \pi_r^\infty$.
\end{itemize}
Moreover for any subquotient $H$ of $H_{d-1}^\mathrm{BM}(X_{\lim,\bullet},\C)$
which is of finite length as a representation of $\GL_d(\A^\infty)$,
the multiplicity of $\pi$ in $H$ is at most one.
\end{thm}
\begin{remark}
Any open compact subgroup of $\GL_d(\A^\infty)$
is conjugate to an open subgroup of $\GL_d(\wh{A})$.
The set of the open subgroups
of $\GL_d(\wh{A})$ is cofinal in the inductive
system of all open compact subgroups of $\GL_d(\A^\infty)$.
Therefore, to prove Theorem~\ref{7_prop1}, we may
without loss of generality assume that 
the group $\bK$ is contained in
$\GL_d(\wh{A})$, and we may
replace the inductive limit $\varinjlim_{\bK}$ in the definition of
$H_{d-1}^\mathrm{BM}(X_{\lim,\bullet},M)$ and $H_{d-1}(X_{\lim,\bullet},M)$
with the inductive limit $\varinjlim_{\bK\subset \GL_d(\wh{A})}$.
\end{remark}

From now on until the end of this section,
we exclusively deal with the subgroups 
$\bK \subset \GL_d(\A^\infty)$ contained in $\GL_d(\wh{A})$.
The notation $\varinjlim_{\bK}$ henceforth means
the inductive limit $\varinjlim_{\bK\subset \GL_d(\wh{A})}$.

\subsection{Chains of locally free $\cO_C$-modules}
\label{sec:locfree}
\ 

In this paragraph, we introduce some terminology
for locally free $\cO_C$-modules of rank $d$ 
and then describe the
sets of simplices of $X_{\bK,\bullet}$
in terms of chains of locally free $\cO_C$-modules 
of rank $d$.
The terminology here will be used in our proof 
of Theorem~\ref{7_prop1}.

Let $\eta : \Spec F \to C$ denote the
generic point of $C$.
For each $g \in \GL_d(\A^\infty)$ and an 
$\cO_\infty$-lattice $L_\infty \subset 
\cO_\infty^{\oplus d}$,
we denote by $\cF[g,L_{\infty}]$ the
$\cO_C$-submodule of $\eta_* F^{\oplus d}$ 
characterized by the following properties:
\begin{itemize}
\item $\cF[g,L_{\infty}]$ is a
locally free $\cO_C$-module of rank $d$.
\item $\Gamma(\Spec A, \cF[g,L_{\infty}])$ 
is equal to the $A$-submodule 
$\wh{A}^{\oplus d} g^{-1} \cap F^{\oplus d}$ of 
$F^{\oplus d} = \Gamma(\Spec A,\eta_* F^{\oplus d})$.
\item Let $\iota_\infty$ denote the morphism $\Spec \cO_\infty \to C$.
Then $\Gamma(\Spec \cO_\infty, \iota_\infty^* \cF[g,L_{\infty}])$
is equal to the $\cO_\infty$-submodule $L_{\infty}$ of
$F_\infty^{\oplus d} = \Gamma(\Spec \cO_\infty, \iota_\infty^* \eta_* F^{\oplus d})$.
\end{itemize}

Let $\cF$ be a locally free $\cO_C$-modules
of rank $d$. Let $I \subset A$ be a
non-zero ideal. We regard the $A$-module
$A/I$ as a coherent $\cO_C$-module of
finite length. A level $I$-structure on
$\cF$ is a surjective homomorphism
$\cF \to (A/I)^{\oplus d}$ of $\cO_C$-modules.
Let $\bK^\infty_I
\subset \GL_d(\wh{A})$ be 
the kernel of the
homomorphism $\GL_d(\wh{A}) \to \GL_d(A/I)$. 
The group $\GL_d(A/I) \cong
\GL_d(\wh{A})/\bK^\infty_I$ acts from
the left on the set of level $I$-structures on $\cF$,
via its left action on $(A/I)^{\oplus d}$.
(We regard $(A/I)^{\oplus d}$ as an $A$-module of
row vectors. The left action of $\GL_d(A/I)$ on
$(A/I)^{\oplus d}$ is described as $g\cdot b
= b g^{-1}$ for $g\in \GL_d(A/I)$, $b\in
(A/I)^{\oplus d}$.) For a subgroup
$\bK \subset \GL_d(\wh{A})$ containing $\bK_I^\infty$,
a level $\bK$-structure on $\cF$ is a
$\bK/\bK^\infty_I$-orbit of level
$I$-structures on $\cF$. For an
open subgroup $\bK \subset \GL_d(\wh{A})$,
the set of level $\bK$-structures on $\cF$
does not depend, up to canonical isomorphisms,
on the choice of an ideal $I$ with
$\bK_I^\infty \subset \bK$.

Let $\bK \subset \GL_d(\wh{A})$ be an open subgroup.
Let $(g,\sigma)$ be an $i$-simplex of
$\wt{X}_{\bK,\bullet}$.
Take a chain $\cdots \supsetneqq L_{-1} 
\supsetneqq L_0 \supsetneqq 
L_1 \supsetneqq \cdots$ of $\cO_{\infty}$-lattices
of $F_{\infty}^{\oplus d}$ which represents $\sigma$.
To $(g,\sigma)$ we associate the
chain $\cdots \supsetneqq \cF[g,L_{-1}] \supsetneqq
\cF[g,L_0] \supsetneqq \cF[g,L_1] \supsetneqq \cdots$ 
of $\cO_C$-submodules of $\eta_* F^{\oplus d}$.
Then the set of $i$-simplices in $\wt{X}_{\bK,\bullet}$
is identified with the set of the equivalences classes
of chains $\cdots \supsetneqq \cF_{-1} \supsetneqq \cF_0 
\supsetneqq \cF_1 \supsetneqq \cdots$ of $\cO_{\infty}$-lattices
of locally free $\cO_C$-submodules of rank $d$ of
$\eta_* \eta^* \cO_C^{\oplus d}$ with a level 
$\bK$-structure 
such that $\cF_{j-i-1}$ equals the
twist $\cF_{j}(\infty)$ as an $\cO_C$-submodule of
$\eta_* F^{\oplus d}$ with a level $\bK$-structure
for every $j\in \Z$.
Two chains $\cdots \supsetneqq \cF_{-1} \supsetneqq
 \cF_0 \supsetneqq \cF_1 \supsetneqq \cdots$ and 
$\cdots \supsetneqq \cF'_{-1} \supsetneqq \cF'_0 
\supsetneqq \cF'_1 \supsetneqq \cdots$ are equivalent 
if and only if there exists an integer $l$ such that
$\cF_{j} = \cF'_{j+l}$ as an $\cO_C$-submodule of
$\eta_* F^{\oplus d}$ with a level structure
for every $j\in \Z$.

Let $g\in \GL_d(\A^\infty)$ and let 
$L_\infty$ be an $\cO_\infty$-lattice of $F_\infty^{\oplus d}$.
For $\gamma \in \GL_d(F)$, the two $\cO_C$-submodules
$\cF[g,L_\infty]$ and $\cF[\gamma g,\gamma L_\infty]$ 
are isomorphic
as $\cO_C$-modules. The set of $i$-simplices
in $X_{\bK,\bullet}$
is identified with the set of the equivalence classes
of chains
$\cdots \inj \cF_{1} \inj  \cF_0 \inj
\cF_{-1} \inj \cdots$ of injective non-isomorphisms
of locally free $\cO_C$-modules of rank $d$ with
a level $\bK$-structure such that the image of
$\cF_{j+i+1}\to \cF_j$ equals the image of
the canonical injection 
$\cF_{j}(-\infty)\inj \cF_j$ for every $j\in \Z$.
Two chains $\cdots \inj \cF_{1} \inj \cF_0 \inj
\cF_{-1} \inj \cdots$ and 
$\cdots \inj \cF'_{1} \inj \cF'_0 \inj
\cF'_{-1} \inj \cdots$ are equivalent if and only if
there exists an integer $l$ and an isomorphism
$\cF_{j} \cong \cF'_{j+l}$ of $\cO_C$-modules with
level structures for every $j\in \Z$ such that 
the diagram
$$
\begin{CD}
\cdots @>>> \cF_{1} @>>> \cF_0 @>>> \cF_{-1} @>>> \cdots \\
@. @V{\cong}VV @V{\cong}VV @V{\cong}VV @. \\
\cdots @>>> \cF'_{l+1} @>>> \cF'_l @>>> \cF'_{l-1} @>>> \cdots
\end{CD}
$$
is commutative.

\subsection{Harder-Narasimhan polygons}
\label{sec:HN}
Let $\cF$ be a locally free $\cO_C$-module of 
rank $r$. For an $\cO_C$-submodule $\cF' \subset \cF$
(note that $\cF'$ is automatically locally free),
we set $z_{\cF}(\cF') = (\rank(\cF'),\deg(\cF'))
\in \Q^2$. It is known that there exists a unique convex, 
piecewise affine,
affine on $[i-1,i]$ for
$i=1,\ldots, r$, 
continuous function 
$p_{\cF} :[0,r] \to \R$ on the interval
$[0,r]$ such that the 
convex hull of the set $\{z_{\cF}(\cF')\ |\ 
\cF' \subset \cF\}$ in $\R^2$ equals
$\{(x,y)\ |\ 0\le x\le r,\, y\le p_{\cF}(x) \}$.
We define the function $\Delta p_{\cF}:
\{1,\ldots,d-1\} \to \R$ as
$\Delta p_{\cF}(i) = 2 p_{\cF}(i) - p_{\cF}(i-1)
- p_{\cF}(i+1)$. Then $\Delta p_{\cF}(i) \ge 0$
for all $i$. We note that for an invertible $\cO_C$-module
$\cL$, $\Delta p_{\cF \otimes \cL}$ equals $\Delta p_{\cF}$.
The theory of Harder-Narasimhan
filtration (\cite{HN}) implies that, 
if $i \in \Supp(\Delta p_{\cF})
= \{ i\ |\ \Delta p_{\cF}(i)>0 \}$, 
then there exists a
unique $\cO_C$-submodule $\cF' \subset \cF$
satisfying $z_{\cF}(\cF')=(i,p_{\cF}(i))$. 
We denote this $\cO_C$-submodule $\cF'$ by $\cF_{(i)}$.
The submodule $\cF_{(i)}$ has the following properties.
\begin{itemize}
\item If $i, j \in \Supp(\Delta p_{\cF})$ with $i\le j$,
then $\cF_{(i)}\subset \cF_{(j)}$ and $\cF_{(j)}/\cF_{(i)}$
is locally free.
\item If $i \in \Supp(\Delta p_{\cF})$, then
$p_{\cF_{(i)}}(x) = p_{\cF}(x)$
for $x \in [0,i]$ and $p_{\cF/\cF_{(i)}}(x-i)
=p_{\cF}(x) - \deg(\cF_{(i)})$ for 
$x \in [i,r]$.
\end{itemize}

\begin{lem}\label{7_diffdeg}
Let $\cF$ be a locally free
$\cO_C$-module of finite rank, and
let $\cF'\subset \cF$ be a $\cO_C$-submodule 
of the same rank. 
Then we have 
$0 \le p_{\cF}(i) -p_{\cF'}(i) \le 
\deg(\cF)-\deg(\cF')$
for $i =1,\ldots,\rank(\cF)-1$.
\end{lem}
\begin{proof}
Immediate from the definition of $p_{\cF}$.
\end{proof}

\subsection{}
\label{sec:props}
In this paragraph, we state two propositions
(Propositions~\ref{7_prop1b} and~\ref{7_prop2}).
The proofs are given in Sections~\ref{7_prop1} and~\ref{7_prop1b}.

Given a subset $\cD \subset \{1,\ldots,d-1\}$
and a real number $\alpha > 0$,
we define the simplicial subcomplex
$X_{\bK,\bullet}^{(\alpha),\cD}$ of
$X_{\bK,\bullet}$ as follows:
A simplex of $X_{\bK,\bullet}$
belongs to $X_{\bK,\bullet}^{(\alpha),\cD}$
if and only if each of its vertices is
represented by a 
locally free $\cO_C$-module $\cF$
of rank $d$ with a level $\bK$-structure
such that $\Delta p_{\cF}(i) \ge \alpha$
holds for every $i \in \cD$.

Let $X_{\bK,\bullet}^{(\alpha)}$ denote
the union $X_{\bK,\bullet}^{(\alpha)}
= \bigcup_{\cD \neq \emptyset} X_{\bK,\bullet}^{(\alpha),\cD}$.

\begin{lem}\label{7_lem:finiteness}
For any $\alpha >0$, the set of the simplices 
in $X_{\bK,\bullet}$ not belonging to 
$X_{\bK,\bullet}^{(\alpha)}$ is finite.
\end{lem}
\begin{proof}
Let $\cP$ denote the set of continuous,
convex functions $p':[0,d]\to \R$ with
$p'(0)=0$ such that $p'(i)\in \Z$ and 
$p'$ is affine on $[i-1,i]$ for $i=1,\ldots,d$.
It is known that for any $r \ge 1$ and $f \in \Z$,
there are only a finite number of isomorphism classes
of semi-stable locally free $\cO_C$-modules of 
rank $r$ with degree $f$. Hence by the theory of
Harder-Narasimhan filtration, for any $p' \in \cP$,
the set of the isomorphism classes
of locally free $\cO_C$-modules $\cF$ with
$p_{\cF} = p'$ is finite.
Let us give an action of the group $\Z$ on the 
set $\cP$, by setting $(a\cdot p')(x)=p'(x)+ a\deg(\infty)x$
for $a \in \Z$ and for $p' \in \cP$.
Then $p_{\cF(a \infty)}= a\cdot p_{\cF}$ for
any $a \in \Z$ and for any locally free $\cO_C$-module
$\cF$ of rank $d$. For $\alpha >0$ let 
$\cP^{(\alpha)} \subset \cP$ denote the set 
of functions $p' \in \cP$ 
with $2p'(i)- p'(i-1)-p'(i+1) \le \alpha$
for each $i \in \{1,\ldots, d-1\}$.
An elementary argument shows
that the quotient $\cP^{(\alpha)}/\Z$ is a finite set,
whence the claim follows.
\end{proof}

Lemma~\ref{7_lem:finiteness} implies that 
$H_{d-1}^\mathrm{BM}(X_{\bK,\bullet},\Q)$ is canonically
isomorphic to the projective limit
$\varprojlim_{\alpha >0}
H_{d-1}(X_{\bK,\bullet},X_{\bK,\bullet}^{(\alpha)}; \Q)$
and $H_c^{d-1}(X_{\bK,\bullet},\Q)$ is canonically
isomorphic to the inductive limit
$\varinjlim_{\alpha >0}
H^{d-1}(X_{\bK,\bullet},X_{\bK,\bullet}^{(\alpha)}; \Q)$.
Thus we have an exact sequence 
\begin{equation}\label{7_cohseq}
\varinjlim_{\alpha >0} 
H^{d-2}(X^{(\alpha)}_{\bK,\bullet},\Q)
\to H_c^{d-1}(X_{\bK,\bullet},\Q)
\to H^{d-1}(X_{\bK,\bullet},\Q)
\to \varinjlim_{\alpha >0} 
H^{d-1}(X^{(\alpha)}_{\bK,\bullet},\Q).
\end{equation}

\begin{prop}\label{7_prop1b}
For $\alpha' \ge \alpha >(d-1)\deg(\infty)$, 
the homomorphisms
$H^*(X^{(\alpha)}_{\bK,\bullet},\Q)
\to H^*(X^{(\alpha')}_{\bK,\bullet},\Q)$
are isomorphisms.
\end{prop}
We give a proof of Proposition~\ref{7_prop1b} in Section~\ref{sec:pf1b}.

\begin{lem}\label{7_betag}
For every $g \in \GL_d(\A^\infty)$
satisfying $g^{-1}\bK g \subset \GL_d(\wh{A})$,
there exists a real number $\beta_g \ge 0$
such that the isomorphism $\xi_g:X_{\bK,\bullet}
\xto{\cong} X_{g^{-1}\bK g,\bullet}$ sends
$X_{\bK,\bullet}^{(\alpha),\cD}$ to
$X_{g^{-1} \bK g,\bullet}^{(\alpha -\beta_g),\cD}
\subset X_{g^{-1}\bK g,\bullet}$
%
for all $\alpha > \beta_g$, for all open compact
$\bK \subset \GL_d(\A^\infty)$, and for 
all nonempty subset $\cD \subset \{1,\ldots,d-1 \}$.

\end{lem}
\begin{proof}
Take two elements $a,b \in \A^{\infty \times}
\cap \wh{A}$ such that both $ag$ and
$bg^{-1}$ lie in $\GL_d(\A^\infty)\cap
\Mat_d(\wh{A})$. Then for any
$h \in \GL_d(\A^\infty)$ we have 
$a
\wh{A}^{\oplus d}h^{-1} 
\subset \wh{A}^{\oplus d}g^{-1} h^{-1} \subset b^{-1} \wh{A}^{\oplus d}h^{-1}$.
This implies that,
for any vertex $x \in X_{\bK,0}$,
if we take suitable representatives 
$\cF_x$, $\cF_{\xi_{g}(x)}$ of the equivalence classes
of locally free $\cO_C$-modules 
corresponding to $x$, $\xi_{g}(x)$, then
there exists a sequence of injections 
$\cF_x(-\divi(a)) \inj \cF_{\xi_g(x)}
\inj \cF_x(\divi(b))$. Applying Lemma~\ref{7_diffdeg},
we see that there exists a positive real number $m_g>0$
not depending on $x$ such that 
$|p_{\cF_x}(i) - p_{\cF_{\xi_{g}(x)}}(i)| < m_g$ for all $i$.
Hence the claim follows.
\end{proof}

Thus the group $\GL_d(\A^\infty)$ acts on 
$\varinjlim_{\bK} \varinjlim_{\alpha > 0} 
H^{*}(X^{(\alpha),\cD}_{\bK,\bullet},\Q)$
in such a way that the exact sequence (\ref{7_cohseq})
is $\GL_d(\A^\infty)$-equivariant.

We use a covering spectral sequence 
\begin{equation}\label{7_specseq}
E_1^{p,q} = \bigoplus_{\sharp \cD = p+1}
H^q(X^{(\alpha),\cD}_{\bK,\bullet},\Q)
\Rightarrow H^{p+q}(X^{(\alpha)}_{\bK,\bullet},\Q)
\end{equation}
with respect to
the covering $X_{\bK,\bullet}^{(\alpha)} = \bigcup_{1\le i \le d-1}
X_{\bK,\bullet}^{(\alpha),\{i\}}$ of $X_{\bK,\bullet}^{(\alpha)}$. 
For $\alpha' \ge \alpha>0$, the inclusion 
$X_{\bK,\bullet}^{(\alpha,\cD)} \to
X_{\bK,\bullet}^{(\alpha',\cD)}$
induces a morphism of spectral sequences.
Taking the inductive limit, we obtain the 
spectral sequence 
$$
E^{p,q}_1
= \bigoplus_{\sharp \cD = p+1}
\varinjlim_{\alpha} H^q(X^{(\alpha),\cD}_{\bK,\bullet},\Q)
\Rightarrow
\varinjlim_{\alpha} H^{p+q}(X^{(\alpha)}_{\bK,\bullet},\Q).
$$
For $g \in \GL_d(\A^\infty)$ satisfying
$g^{-1}\bK g \subset \GL_d(\wh{A})$, let $\beta_g$ be as
in Lemma~\ref{7_betag}. Then for $\alpha >\beta_g$ 
the isomorphism $\xi_g :X_{\bK,\bullet}\xto{\cong}
X_{g\bK g^{-1},\bullet}$ induces a homomorphism
from the spectral sequence (\ref{7_specseq}) for 
$X_{\bK,\bullet}^{(\alpha)}$ to that 
for $X_{\bK,\bullet}^{(\alpha -\beta_g)}$. 
Passing to the inductive limit with respect to $\alpha$
and then passing to the inductive limit with respect to 
$\bK$, we obtain the left action of the group 
$\GL_d(\A^\infty)$ on the spectral sequence
\begin{equation}\label{7_limspecseq}
E^{p,q}_1
= \bigoplus_{\sharp \cD = p+1}
\varinjlim_{\bK} 
\varinjlim_{\alpha} H^q(X^{(\alpha),\cD}_{\bK,\bullet},\Q)
\Rightarrow
\varinjlim_{\bK} 
\varinjlim_{\alpha} H^{p+q}(X^{(\alpha)}_{\bK,\bullet},\Q).
\end{equation}

For a subset $\cD$ of $\{1,\ldots,d-1\}$,
we define the algebraic
groups $P_{\cD}$, $N_{\cD}$ and
$M_{\cD}$ as follows. We write 
$\cD= \{i_1,\ldots,i_{r-1} \}$, with
$i_0=0 < i_1 < \cdots < i_{r-1} <i_r=d$ 
and set $d_j = i_j -i_{j-1}$ for
$j=1,\ldots,r$.
We define $P_{\cD}$, $N_{\cD}$ and
$M_{\cD}$ as the standard parabolic
subgroup of $\GL_d$ of type
$(d_1,\ldots,d_r)$, the unipotent radical
of $P_{\cD}$, and the quotient group $P_{\cD}/N_{\cD}$
respectively. We identify the group $M_{\cD}$ 
with $\GL_{d_1}\times\cdots \times \GL_{d_r}$.

\begin{prop}\label{7_prop2}
Let the notations be above.
Then as a smooth $\GL_d(\A^\infty)$-module,
$\varinjlim_{\bK} \varinjlim_{\alpha >0}
H^*(X_{\bK,\bullet}^{(\alpha),\cD},\Q)$ is
isomorphic to 
$$
\Ind_{P_{\cD}(\A^\infty)}^{\GL_d(\A^\infty)}
\bigotimes_{j=1}^r 
\varinjlim_{\bK_j \subset \GL_{d_j}(\wh{A})} 
H^*(X_{\GL_{d_j},\bK_j,\bullet},\Q),
$$
where the group $P_{\cD}(\A^\infty)$ acts on
${\displaystyle \bigotimes_{j=1}^r 
\varinjlim_{\bK_j\subset \GL_{d_j}(\wh{A})}} 
H^*(X_{\GL_{d_j},\bK_j,\bullet},\Q)$
via the quotient 
$P_{\cD}(\A^\infty) \to M_{\cD}(A^\infty)
= \prod_j \GL_{d_j}(\A^\infty)$,
and $\Ind_{P_{\cD}(\A^\infty)}^{\GL_d(\A^\infty)}$
denotes the parabolic induction 
unnormalized by the modulus function.
\end{prop}
The proof will be given in Section~\ref{sec:pf2}.

\subsection{Proof of Theorem~\ref{7_prop1}}
\label{sec:pfthm}
Here we give a proof of Theorem~\ref{7_prop1}
assuming Propositions~\ref{7_prop1b} and~\ref{7_prop2}.
\begin{proof}[Proof of Theorem~\ref{7_prop1}]
Let $\bK, \bK' \subset \GL_d(\A^\infty)$ be two
compact open subgroups with $\bK' \subset \bK$.
The pull-back morphism from the cochain complex
of $X_{\bK,\bullet}$ to that of $X_{\bK',\bullet}$
preserves the cochains with finite supports.
Thus we have pull-back homomorphisms 
$H^*_c(X_{\bK,\bullet},\Q) \to H^*_c(X_{\bK',\bullet},\Q)$
which is compatible with
the usual pull-back homomorphism 
$H^*(X_{\bK,\bullet},\Q) \to H^*(X_{\bK',\bullet},\Q)$.
For an abelian group $M$, we let
$H^*(X_{\lim,\bullet},M)=H^*(X_{\GL_d,\lim,\bullet},M)$ and
$H_c^*(X_{\lim,\bullet},M)=H_c^*(X_{\GL_d,\lim,\bullet},M)$
denote the
inductive limits
$\varinjlim_{\bK}H^*(X_{\bK,\bullet},M)$ and
$\varinjlim_{\bK}H_c^*(X_{\bK,\bullet},M)$,
respectively.
If $M$ is a $\Q$-vector space, then
for each compact open subgroup 
$\bK \subset \GL_d(\A^\infty)$,
the homomorphism
$H^*(X_{\bK,\bullet},M) \to H^*(X_{\lim,\bullet},M)$
is injective and its image is equal to the
$\bK$-invariant part $H^*(X_{\lim,\bullet},M)^{\bK}$
of $H^*(X_{\lim,\bullet},M)$.
Similar statement holds for $H_c^*$.
It follows from Lemma~\ref{7_lem1} that
the inductive limits $H^{d-1}(X_{\lim,\bullet},\Q)$
and $H_c^{d-1}(X_{\lim,\bullet},\Q)$ are admissible
$\GL_d(\A^\infty)$-modules, and are isomorphic to
the contragradient of $H_{d-1}(X_{\lim,\bullet},\Q)$
and $H_{d-1}^\mathrm{BM}(X_{\lim,\bullet},\Q)$,
respectively.
Since $\St_d$ is self-contragradient,
it follows from the compatibility of
the normalized parabolic inductions with taking contragradient
that it suffice to prove that any irreducible subquotient of
$H^{d-1}_c(X_{\lim,\bullet},\C)$
satisfies the properties in the statement of Theorem~\ref{7_prop1}.
Let $\pi$ be an irreducible subquotient of
$H^{d-1}_c(X_{\lim,\bullet},\C)$.
Then Proposition~\ref{7_prop2} combined with 
the spectral sequence (\ref{7_limspecseq}) shows that
there exists a subset $\cD \subset \{1,\ldots,d-1\}$
such that $\pi^\infty$ is isomorphic to a subquotient of
$\Ind_{P_{\cD}(\A^\infty)}^{GL_d(\A^\infty)}
\bigotimes_{j=1}^r \varinjlim_{\bK_j}
H^{d_j -1}(X_{\GL_{d_j},\bK_j,\bullet},\C)$.
Here $r=\sharp \cD +1$, and $d_1,\ldots,d_{r} \ge 1$ 
are the integers satisfying $\cD=\{d_1,d_1+d_2,\ldots,
d_1+\cdots+d_{r-1} \}$ and $d_1 +\cdots + d_{r} =d$.
By Proposition~\ref{prop:66_3}, $\pi^\infty$ is isomorphic 
to a subquotient of the non-$\infty$-component
of the induced representation from $P_{\cD}(\A)$ to
$\GL_d(\A)$ of an irreducible cuspidal automorphic representation
$\pi_1 \otimes \cdots \otimes \pi_r$ of $M_{\cD}(\A)$ whose component 
at $\infty$ is isomorphic to
the tensor product the Steinberg representations. 

It remains to prove the claim of the multiplicity.
The Ramanujan-Petersson conjecture proved by Lafforgue
shows that each place $v$ of $F$, the representation $\pi_{i,v}$
is tempered. Hence for almost all places $v$ of $F$,
the representation $\pi_v$ of $\GL_d(F_v)$ is unramified
and its associated Satake parameters $\alpha_{v,1},\cdots ,\alpha_{v,d}$
have the following property: for each $i$ with $1 \le i \le r$,
exactly $d_i$ parameters of $\alpha_{v,1},\cdots ,\alpha_{v,d}$
have the complex absolute value $q_v^{a_i/2}$ where
$q_v$ denotes the cardinality of the residue field at $v$
and $a_i = \sum_{i<j\le r} d_j - \sum_{1 \le j <i} d_j$.
This shows that the subset $\cD$ is uniquely determined by $\pi$.
It follows from the multiplicity one theorem and the strong multiplicity
one theorem that the
cuspidal automorphic representation 
$\pi_1 \otimes \cdots \otimes \pi_r$ of 
$M_{\cD}(\A)$
is also uniquely determined by $\pi$. 
Hence it suffices to show that
the representation $\Ind_{P_{\cD}(F_v)}^{\GL_d(F_v)} 
\pi_{1,v} \otimes \cdots \otimes \pi_{r,v}$ of $\GL_d(F)$ is of
multiplicity free for every place $v$ of $F$.
For $1 \le i \le r$, let $\Delta_i$ denote the multiset of
segments corresponding to the representation 
$\pi_{i,v}\otimes |\det(\ )|_v^{a_i/2}$
in the sense of \cite{Zelevinsky}. We denote by $\Delta_i^t$
the Zelevinski dual of $\Delta_i$.
Let $i_1, i_2$ be integers with $1 \le i_1 < i_2 \le r$
and suppose that there exist a segment in $\Delta_{i_1}^t$ 
and a segment in $\Delta_{i_2}^t$ which are linked.
Since $\pi_{i_1,v}$ and $\pi_{i_2,v}$ are tempered,
it follows that $i_2 = i_1 +1$ and that there exists a character 
$\chi$ of $F_v^\times$ such that both $\pi_{i_1,v} \otimes \chi$
and $\pi_{i_2,v} \otimes \chi$ are the Steinberg representations.
In this case the multiset $\Delta_{i_j}^t$ consists of
a single segment for $j=1,2$ and the
the unique segment in $\Delta_{i_1}^t$ and the
unique segment in $\Delta_{i_2}^t$ are juxtaposed.
Thus the claim is obtained by applying the formula in 
\cite[9.13, Proposition, p.201]{Zelevinsky}.
\end{proof}

\subsection{Proof of Proposition~\ref{7_prop1b}}
\label{sec:pf1b}
We need some preparation.
\begin{lem}\label{7_HNdiff}
Let $\cF$ be a locally free $\cO_C$-module
of rank $d$. Let $\cF' \subset \cF$ be
a n $\cO_C$-submodule of the same rank. 
Suppose that 
$\Delta p_{\cF}(i) > \deg(\cF) 
-\deg(\cF')$. Then we have 
$\cF'_{(i)} = \cF_{(i)} \cap \cF'$.
\end{lem}

\begin{proof}
It suffices to prove that
$\cF'_{(i)}\subset \cF_{(i)}$.
Assume otherwise. Let us consider
the short exact sequence
$$
0\to \cF'_{(i)}\cap
\cF_{(i)} \to \cF'_{(i)} \to
\cF'_{(i)}/(\cF'_{(i)}\cap
\cF_{(i)}) \to 0
$$
Let $r$ denote the rank of
$\cF'_{(i)}\cap \cF_{(i)}$. 
By assumption, $r$ is strictly
smaller than $i$. Hence
$$
\begin{array}{rl}
\deg(\cF'_{(i)}) & =
\deg(\cF'_{(i)}\cap
\cF_{(i)}) + 
\deg(\cF'_{(i)}/(\cF'_{(i)}\cap
\cF_{(i)})) \\
& \le p_{\cF}(r) + 
p_{\cF/\cF_{(i)}}(i-r) \\
& \le p_{\cF}(i) - (i-r)(p_{\cF}(i)-p_{\cF}(i-1))
+ (i-r)(p_{\cF}(i+1) -p_{\cF}(i)) \\
& = \deg(\cF_{(i)}) - (i-r) \Delta p_{\cF}(i) \\
& < \deg(\cF_{(i)}) - (\deg(\cF)-\deg(\cF')).
\end{array}
$$
On the other hand, Lemma~\ref{7_diffdeg}
shows that $\deg(\cF_{(i)}\cap \cF')\ge
\deg(\cF_{(i)}) - (\deg(\cF)-\deg(\cF'))$.
This is a contradiction.
\end{proof}

Let $\Flag_{\cD}$ denote the set
$$
\Flag_{\cD} = \{ f=
[0 \subset V_1 \subset \cdots \subset
V_{r-1} \subset F^{\oplus d}] \ |\ \dim(V_j)=i_j \}
$$
of flags in $F^{\oplus d}$.

Let $\wt{X}_{\bK,\bullet}^{(\alpha),\cD}$ denote
the inverse image of $X_{\bK,\bullet}^{(\alpha),\cD}$
by the morphism $\wt{X}_{\bK,\bullet}
\to X_{\bK,\bullet}$. 
For $f = [0 \subset V_1 \subset \cdots \subset
V_{r-1} \subset F^{\oplus d}]\in \Flag_{\cD}$, let 
$\wt{X}_{\bK,\bullet}^{(\alpha),\cD,f}$ 
denote the simplicial subcomplex of 
$\wt{X}_{\bK,\bullet}^{(\alpha),\cD}$
consisting of the simplices in $\wt{X}_{\bK,\bullet}$ 
whose representative
$\cdots \supsetneqq \cF_{-1} \supsetneqq 
\cF_0 \supsetneqq \cF_1
\supsetneqq \cdots$ satisfies
$\cF_{l,(i_j)} = \cF_l \cap \eta_* V_{i_j}$
for every $l\in\Z$, $j=1,\ldots,r-1$.
Lemma~\ref{7_HNdiff} implies that, for 
$\alpha > (d-1)\deg(\infty)$,
$\wt{X}_{\bK,\bullet}^{(\alpha),\cD}$ is decomposed
into a disjoint union $\wt{X}_{\bK,\bullet}^{(\alpha),\cD}
= \coprod_{f \in \Flag_{\cD}}
\wt{X}_{\bK,\bullet}^{(\alpha),\cD,f}$.
An argument similar to that in the proof of
Lemma~\ref{7_betag} shows that, 
for each $g \in \GL_d(\A^\infty)$ satisfying
$g^{-1}\bK g \subset \GL_d(\wh{A})$, there exists
a real number $\beta'_g > \beta_g$ such that
the isomorphism $\wt{\xi}_g$
sends $\wt{X}_{\bK,\bullet}^{(\alpha),\cD,f}$ to
$\wt{X}_{g \bK g^{-1},\bullet}^{(\alpha-\beta_g),\cD,f}
\subset \wt{X}_{g \bK g^{-1},\bullet}$
for $\alpha > \beta'_g$ and for any $f \in \Flag_{\cD}$.

For $\gamma \in \GL_d(F)$, the action of
$\gamma$ on $\wt{X}_{\bK,\bullet}$
sends $\wt{X}_{\bK,\bullet}^{(\alpha),\cD,f}$ 
bijectively to 
$\wt{X}_{\bK,\bullet}^{(\alpha),\cD,\gamma f}$.
Let $f_0 = [0\ \subset F^{\oplus i_1}\oplus 
\{0\}^{\oplus d-i_1} 
\subset \cdots \subset F^{\oplus i_{r-1}}
\oplus \{0\}^{\oplus d -i_{r-1}}
\subset F^{\oplus d}]\in \Flag_{\cD}$ be
the standard flag. The group $\GL_d(F)$
acts transitively on $\Flag_{\cD}$ and its
stabilizer at $f_0$ equals $P_{\cD}(F)$.
Hence for $\alpha > (d-1)\deg(\infty)$,

$X_{\bK,\bullet}^{(\alpha),\cD}$ is
isomorphic to the quotient 
$P_{\cD}(F)\backslash 
\wt{X}_{\bK,\bullet}^{(\alpha),\cD,f_0}$.

For $g \in \GL_d(\A^\infty)$, we set 
$$
\wt{Y}_{\bK,\bullet}^{(\alpha),\cD,g} =
\wt{X}_{\bK,\bullet}^{(\alpha),\cD,f_0}
\cap (P_{\cD}(\A^\infty)g/(g^{-1}P_{\cD}(\A^\infty)g \cap \bK)
\times \cBT_{\bullet})
$$
and $Y_{\bK,\bullet}^{(\alpha),\cD,g} = 
P_{\cD}(F)\backslash \wt{Y}_{\bK,\bullet}^{(\alpha),\cD,g}$.
We omit the superscript $g$ on 
$\wt{Y}_{\bK,\bullet}^{(\alpha),\cD,g}$
and $Y_{\bK,\bullet}^{(\alpha),\cD,g}$ if $g=1$.
If we take a complete set $T \subset \GL_d(\A^\infty)$ 
of representatives of $P_{\cD}(\A^\infty)\backslash \GL_d(\A^\infty)$, 
then we have $\wt{X}_{\bK,\bullet}^{(\alpha),\cD,f_0} = 
\coprod_{g \in T} \wt{Y}_{\bK,\bullet}^{(\alpha),\cD,g}$.
For each $g \in \GL_d(\A^\infty)$ satisfying
$g^{-1}\bK g \subset \GL_d(\wh{A})$, there exists
a real number $\beta'_g > \beta_g$ such that
the isomorphism $\wt{\xi}_g$
sends $\wt{Y}_{\bK,\bullet}^{(\alpha),\cD,g'}$ 
to $\wt{Y}_{g \bK g^{-1},\bullet}^{(\alpha-\beta_g),\cD,g'g}
\subset \wt{X}_{g \bK g^{-1},\bullet}$
for $\alpha > \beta'_g$, for any $f \in \Flag_{\cD}$
and for any $g' \in \GL_d(\A^\infty)$.
Hence, as a smooth $\GL_d(\A^\infty)$-module,
$\varinjlim_{\bK} \varinjlim_{\alpha}
H^*(X_{\bK,\bullet}^{(\alpha),\cD},\Q)$ is isomorphic to 
$\Ind_{P_{\cD}(\A^\infty)}^{\GL_d(\A^\infty)}
\varinjlim_{\bK} \varinjlim_{\alpha} 
H^* (Y_{\bK,\bullet}^{(\alpha),\cD},\Q)$.
\begin{lem}\label{7_contract}
For any $g \in \GL_d(\A^\infty)$, 
the simplicial complex
$\wt{X}_{\bK,\bullet}^{(\alpha),\cD,f_0}
\cap (\{ g\bK \}\times \cBT_{\bullet})$ is 
non-empty and contractible.
\end{lem}
\begin{proof}
Since $\wt{X}_{\bK,\bullet}^{(\alpha),\cD,f_0}
\cap (\{ g\bK \}\times \cBT_{\bullet})$
is isomorphic to $\wt{X}_{\GL_d(\wh{A}),\bullet}^{(\alpha),\cD,f_0}
\cap (\{ g\GL_d(\wh{A}) \}\times \cBT_{\bullet})$,
we may assume that $\bK = \GL_d(\wh{A})$.
We set $X=\wt{X}_{\GL_d,\GL_d(\wh{A}),\bullet}^{(\alpha),\cD,f_0}
\cap (\{ g\GL_d(\wh{A}) \}\times \cBT_{\GL_d,\bullet})$.

We proceed by induction on $d$, 
in a manner similar to that in the proof of Theorem~4.1 of 
\cite{Gra}.
Let $i\in \cD$ be the minimal element and set $d'=d-i$.
We define the subset $\cD' \subset \{1,\ldots,d'-1 \}$
as $\cD' = \{ i' -i\ |\ i' \in \cD, i'\neq i \}$.
We define $f'_0 \in \Flag_{\cD'}$ as the image of
the flag $f_0$ in $F^{\oplus d}$ with respect to the
the projection $F^{\oplus d} \surj 
F^{\oplus d}/(F^{\oplus i}\oplus \{0\}^{\oplus d'})
\cong F^{\oplus d'}$.
Take an element $g' \in \GL_{d'}(\A^\infty)$ such that
the quotient $\wh{A}^{\oplus d} g^{-1}/
(\wh{A}^{\oplus d}g^{-1}\cap 
(\A^{\infty \oplus i} \oplus \{0\}^{\oplus d'}))$
equals $\wh{A}^{\oplus d'}g^{\prime -1}$ as
an $\wh{A}$-lattice of $\A^{\infty \oplus d'}$.
We set $X'=\wt{X}_{\GL_{d'}, 
\GL_{d'}(\wh{A}),\bullet}^{(\alpha),\cD',f'_0}
\cap (\{ g' \GL_{d'}(\wh{A}) \}\times \cBT_{\GL_{d'},\bullet})$
if $\cD'$ is non-empty. Otherwise we set $X'
= \wt{X}_{\GL_{d'}, \GL_{d'}(\wh{A}),\bullet}
\cap (\{ g' \GL_{d'}(\wh{A}) \}\times \cBT_{\GL_{d'},\bullet})$.
By induction hypothesis, $|X'|$ is contractible.
There is a canonical morphism
$h:X \to X'$ which sends an $\cO_C$-submodule $\cF[g,L_\infty]$ of
$\eta_* F^{\oplus d}$ to the 
$\cO_C$-submodule $\cF[g,L_\infty]/\cF[g,L_\infty]_{(i)}$
of $\eta_* F^{\oplus d'}$.
Let $\epsilon : \Vertex(X) \to \Z$ and
$\epsilon' : \Vertex(X') \to \Z$ denote the maps which
send a locally free $\cO_C$-module $\cF$ to the integer
$[p_{\cF}(1)/\deg(\infty)]$.
We fix an
$\cO_C$-submodule $\cF_0$ of $\eta_* F^{\oplus d}$
whose equivalence class belongs to $X$.
By twisting $\cF_0$ by some power of $\cO_C(\infty)$
if necessary, we may assume that 
$p_{\cF_0}(i) - p_{\cF_0}(i-1) > \alpha$.
We fix a splitting $\cF_0 = \cF_{0,(i)}\oplus \cF'_0$.
This splitting induces an isomorphism
$\varphi : \eta_* \eta^* \cF'_0 \cong \eta_* F^{\oplus d}$.
Let $h' : X' \to X$ denote the morphism which sends
an $\cO_C$-submodule $\cF'$ of
$\eta_* \eta^* F^{\oplus d'}$ to the
$\cO_C$-submodule $\cF_{0,(i)}(\epsilon'(\cF')\infty) 
\oplus \varphi^{-1}(\cF')$ of $\eta_* F^{\oplus d}$.
For each $n \in \Z$, define a morphism
$G_n : X \to X$ by
sending an $\cO_C$-submodule $\cF$ of
$\eta_* \eta^* F^{\oplus d}$ to the
$\cO_C$-submodule $\cF_{0,(i)}((n+\epsilon(\cF))\infty) 
+ \cF$ of $\eta_* F^{\oplus d}$. Then the
argument in~\cite[p. 85--86]{Gra} shows that
$f$ and $|h'|\circ|h| \circ f$ are homotopic
for any map $f:Z \to |X|$ from a compact space $Z$
to $|X|$. Since the map $|h'|\circ|h| \circ f$ factors
through the contractible space $|X'|$, $f$ is
null-homotopic. Hence $|X|$ is contractible.
\end{proof}

\begin{proof}[Proof of Proposition~\ref{7_prop1b}]
For any simplex $\sigma$ in $\wt{X}_{\bK,\bullet}$,
the isotropy group $\Gamma_\sigma \subset \GL_d(F)$
is finite, as remarked in Section~\ref{sec:71}.
Hence by Lemma~\ref{7_contract}, both 
$H^*(Y_{\bK,\bullet}^{(\alpha),\cD,g},\Q)$ and
$H^*(Y_{\bK,\bullet}^{(\alpha'),\cD,g},\Q)$ are canonically
isomorphic to the same group $H^*(P_{\cD}(F),
\Map(P_{\cD}(\A)g/(g^{-1}P_{\cD}(\A^\infty)g \cap \bK),\Q))$ 
for any non-empty subset $\cD \subset \{1,\ldots, d-1\}$
and for $g \in \GL_d(\A^\infty)$.
This shows that $H^*(X_{\bK,\bullet}^{(\alpha),\cD},\Q)
\to H^*(X_{\bK,\bullet}^{(\alpha'),\cD},\Q)$ is an isomorphism.
Since the homomorphisms between the $E_1$-terms of the
spectral sequences (\ref{7_specseq}) for $\alpha$ and 
for $\alpha'$ is an isomorphism, $H^*(X_{\bK,\bullet}^{(\alpha)},\Q)
\to H^*(X_{\bK,\bullet}^{(\alpha')},\Q)$ is an
isomorphism.
\end{proof}

\subsection{Proof of Proposition \ref{7_prop2}}
\label{sec:pf2}
For $j=1,\ldots,r$, let $\bK_j 
\subset \GL_{d_i}(\A^\infty)$ denote the
image of $\bK \cap P_{\cD}(\A^\infty)$ by the composition
$P_{\cD}(\A^\infty) \to M_{\cD}(\A^\infty)
\to \GL_{d_i}(\A^\infty)$.

We define the continuous map
$\wt{\pi}_{\cD,j}: 
|\wt{Y}_{\bK,\bullet}^{(\alpha),\cD}| \to 
|\wt{X}_{\GL_{d_j},\bK_j,\bullet}|$
of topological spaces in the following way.
Let $\sigma$ be an $i$-simplex in 
$\wt{Y}_{\bK,\bullet}^{(\alpha),\cD}$.
Take a chain 
$\cdots \supsetneqq \cF_{-1} \supsetneqq \cF_0
\supsetneqq \cF_1 \supsetneqq \cdots$
of $\cO_C$-modules representing $\sigma$.
For $l\in \Z$ we set $\cF_{l,j} = \cF_{l,(i_j)}/\cF_{l,(i_{j-1})}$,
which is an $\cO_C$-submodule of $\eta_* F^{\oplus d_j}$.
We set $S_j = \{ l \in \Z \ |\ \cF_{l,j}\neq \cF_{l+1,j} \}$.
Define the map $\psi_j : \Z \to S_j$ as
$\psi_j(l) = \min \{ l'\ge l\ |\ l' \in S_j \}$.
Take an order-preserving bijection 
$\varphi_j :S_j \xto{\cong} \Z$.
For $l \in\Z$ set $\cF'_{l} =\cF_{\varphi_j^{-1}(l), j}$.
Then the chain $\cdots \supsetneqq \cF'_{-1} \supsetneqq \cF'_0
\supsetneqq \cF'_1 \supsetneqq \cdots$ defines a
simplex $\sigma'$ in $\wt{X}_{\GL_{d_j},\bK_j,\bullet}$.
We define a continuous map $|\sigma| \to |\sigma'|$
as the affine map sending the vertex of $\sigma$ 
corresponding to $\cF_l$ to the vertex of $\sigma'$
corresponding to $\cF'_{\varphi_j \circ \psi_j(l)}$.
Gluing these maps, we obtain a continuous map
$\wt{\pi}_{\cD,j}: 
|\wt{Y}_{\bK,\bullet}^{(\alpha),\cD}| \to 
|\wt{X}_{\GL_{d_j},\bK_j,\bullet}|$.
We set $\wt{\pi}_{\cD} =
(\wt{\pi}_{\cD,1},\ldots,\wt{\pi}_{\cD,r})
:|\wt{Y}_{\bK,\bullet}^{(\alpha),\cD}| \to 
\prod_{j=1}^{r} 
|\wt{X}_{\GL_{d_j},\bK_j,\bullet}|$.
This continuous map descends to the 
continuous map 
$\pi_{\cD}: |Y_{\bK,\bullet}^{(\alpha),\cD}| \to 
\prod_{j=1}^{r} |X_{\GL_{d_j},\bK_j,\bullet}|$.

If $g \in P_{\cD}(\A^\infty)$ and
$g^{-1}\bK g \subset \GL_d(\wh{A})$, then the isomorphism
$\xi_g :X_{\bK,\bullet} \xto{\cong}
X_{g^{-1}\bK g,\bullet}$ sends 
$Y_{\bK,\bullet}^{(\alpha),\cD}$ inside
$Y_{g^{-1}\bK g,\bullet}^{(\alpha -\beta_g),\cD}$.
If we denote by $(g_1,\ldots,g_r)$ 
the image of $g$ in $M_{\cD}(\A^\infty)
= \prod_{j=1}^r \GL_{d_j}(\A^\infty)$,
then the diagram
$$
\begin{CD}
|Y_{\bK,\bullet}^{(\alpha),\cD}| @>{\xi_g}>>
|Y_{g^{-1}\bK g,\bullet}^{(\alpha-\beta_g),\cD}| \\
@V{\pi_{\cD}}VV @V{\pi_{\cD}}VV \\
\prod_{j=1}^r |X_{\GL_{d_j},\bK_j,\bullet}|
@>{(\xi_{g_1},\ldots,\xi_{g_r})}>> \prod_{j=1}^r
|X_{\GL_{d_j},g_j^{-1} \bK_j g_j,\bullet}|
\end{CD}
$$
is commutative.

With the notations as above, 
suppose that the open compact
subgroup $\bK \subset \GL_d(\A^\infty)$
has the following property.
\begin{equation} \label{7_property}
\text{the homomorphism }
P_{\cD}(\A^\infty)\cap \bK \to
\bK_1 \times \cdots \times \bK_r
\text{ is surjective.}
\end{equation}
For a simplicial complex $X$, we set
$I_X = \Map(\pi_0(X),\Q)$, where $\pi_0(X)$ is 
the set of the connected components
of $X$. Let us consider the following commutative
diagram.
\begin{equation}\label{7_CD}
\begin{CD}
H^*(M_{\cD}(F), 
\Map(\prod_{j=1}^r \pi_0(X_{\GL_{d_j},\bK_j,\bullet}),\Q))
@>>>
H^*(P_{\cD}(F), I_{Y_{\bK,\bullet}^{(\alpha),\cD}})
\\
@VVV @VVV \\
H^*_{M_{\cD}(F)}
(\prod_{j=1}^r |\wt{X}_{\GL_{d_j},\bK_j,\bullet}|,\Q)
@>>> H^*_{P_{\cD}(F)}(|Y_{\bK,\bullet}^{(\alpha),\cD}|,\Q)\\
@AAA @AAA \\
H^*(\prod_{j=1}^r |X_{\GL_{d_j},\bK_j,\bullet}|,\Q)
@>>> H^*(|Y_{\bK,\bullet}^{(\alpha),\cD}|,\Q).
\end{CD}
\end{equation}
Here $H^*_{M_{\cD}(F)}$ and $H^*_{P_{\cD}(F)}$ denote
the equivariant cohomology groups.

\begin{prop}\label{7_prop3}
All homomorphisms in the above
diagram (\ref{7_CD}) are isomorphisms.
\end{prop}

\begin{proof}
We prove that the upper horizontal arrow
and the four vertical arrows are isomorphisms.

First we consider the upper horizontal arrow.
\begin{lem}\label{thislemma}
For $q\ge 1$, the group
$H^q (N_{\cD}(F), I_{Y_{\bK,\bullet}^{(\alpha),\cD}})$
is zero.
\end{lem}
\begin{proof}[Proof of Lemma \ref{thislemma}]
For each $x \in N_{\cD}(F) \backslash 
\pi_0(Y_{\bK,\bullet}^{(\alpha),\cD})$,
take a lift $\wt{x} \in Y_{\bK,\bullet}^{(\alpha),\cD}$
of $x$ and let $N_x \subset N_{\cD}(F)$ denote the stabilizer
of $\wt{x}$. Then the group
$H^*(N_{\cD}(F), I_{Y_{\bK,\bullet}^{(\alpha),\cD}})$
is isomorphic to the direct product 
$$\prod_{x \in N_{\cD}(F) \backslash 
\pi_0(Y_{\bK,\bullet}^{(\alpha),\cD})}
H^*(N_x ,\Q).$$ We note that the group $N_{\cD}(F)$ is
a union $N_{\cD}(F) = \bigcup_{i} U_i$ of 
finite subgroups of $p$-power order where $p$ is 
the characteristic of $F$.  (This follows easily from
\cite[p.2, 1.A.2 Lemma]{KeWe} or from
\cite[p.60, 1.L.1 Theorem]{KeWe}.)
Hence $N_x = \bigcup_{i} (U_i \cap N_x)$. 
Therefore, for an $N_x$-module $M$,
$H^{j}(N_x,M) =0$ for $j\ge 1$ if
the projective system $(H^0(U_i \cap N_x ,M))_i$ 
satisfies the Mittag-Leffler condition. In particular
we have $H^{j}(N_x,\Q) =0$ for $j\ge 1$.
Hence the claim follows.
\end{proof}
We note that $\pi_0(X_{\GL_{d_j},\bK_j,\bullet})$ is
canonically isomorphic to
$\GL_{d_i}(\A^\infty)/\bK_j$
for $j=1,\ldots,r$, and Lemma~\ref{7_contract} implies
that $I_{Y_{\bK,\bullet}^{(\alpha),\cD}}$ is
canonically isomorphic to
$\Map(P_{\cD}(\A^\infty)/
(P_{\cD}(\A^\infty) \cap \bK),\Q)$.
Since $N_{\cD}(F)$ is dense in 
$N_{\cD}(\A^\infty)$, the group
$H^0(N_{\cD}(F), I_{Y_{\bK,\bullet}^{(\alpha),\cD}})$
is canonically isomorphic to the group
$\Map(M_{\cD}(\A^\infty)/\prod_{j=1}^r \bK_j,\Q)$.
Hence the upper horizontal arrow
of the diagram (\ref{7_CD}) is an isomorphism.

Next we consider the two vertical arrows.
Each connected component of
$\wt{X}_{\GL_{d_j},g_j^{-1} \bK_j g_j,\bullet}$
is contractible since it is isomorphic to 
the Bruhat-Tits building for $\GL_{d_j}$.
Recall that the simplicial complex
$X_{\GL_{d_j},g_j^{-1} \bK_j g_j,\bullet}$
is the quotient of 
$\wt{X}_{\GL_{d_j},g_j^{-1} \bK_j g_j,\bullet}$
by the action of $\GL_{d_j}(F)$.
For any simplex $\sigma$ in 
$\wt{X}_{\GL_{d_j},g_j^{-1} \bK_j g_j,\bullet}$,
the isotropy group $\Gamma_\sigma \subset \GL_{d_j}(F)$
of $\sigma$ is finite, as remarked in Section~\ref{sec:71}.
Hence the left two vertical arrows in the diagram (\ref{7_CD})
are isomorphisms.
Similarly, bijectivity of the two right
vertical arrows in the diagram (\ref{7_CD})
follows from Lemma~\ref{7_contract}. Thus we have
a proof of Proposition~\ref{7_prop3}.
\end{proof}

\begin{proof}[Proof of Proposition~\ref{7_prop2}]
Let us consider the lower horizontal arrow
in the diagram (\ref{7_CD}). By Proposition~\ref{7_prop3} 
it is an isomorphism. We note that the compact
open subgroups $\bK \subset \GL_d(\wh{A})$
with property (\ref{7_property}) form a 
cofinal subsystem of the
inductive system of all open compact subgroups
of $\GL_d(\A^\infty)$. Therefore,
passing to the inductive limits with respect to
$\alpha$ and $\bK$ with property (\ref{7_property}), we have
$\varinjlim_\bK \varinjlim_\alpha
H^*(Y_{\bK,\bullet}^{(\alpha),\cD},\Q)
\cong \bigotimes_{j=1}^r \varinjlim_{\bK_j} 
H^*(X_{\GL_{d_j},\bK_j,\bullet},\Q)$ as desired.
\end{proof}

\vspace{1cm}
Satoshi Kondo  \\
Kavli Institute for the Physics and Mathematics of the Universe \\
The University of Tokyo \\
5-1-5 Kashiwanoha, Kashiwa, 277-8583, Japan \\
Tel.: +81-4-7136-4940\\
Fax: +81-4-7136-4941\\
satoshi.kondo@ipmu.jp             \\
          
\noindent Seidai Yasuda \\
              Department of Mathematics, Graduate School of Science, Osaka University \\
s-yasuda@math.sci.osaka-u.ac.jp


\begin{thebibliography}{Har-aaa}
\bibitem[Ab-Br]{Ab-Br}
Abramenko, P., Brown, K.\ S.: 
Buildings.\ Theory and applications.\ 
Graduate Texts in Mathematics 248.\ 
Springer, Berlin (2008)

\bibitem[As-Ru]{AR} Ash, A., Rudolph,  L.:
The modular symbol and continued fractions 
in higher dimensions.\
Invent.\ Math.\ 55,  
no.\ 3, 241--250 (1979)


\bibitem[Bo]{Borel}Borel, A.:
Admissible representations of a semi-simple group
over a local field with vectors fixed under an Iwahori
subgroup.\
Invent.\ Math.\ 
35, 233--259 (1976)

\bibitem[Bo-Ja]{BJ} Borel, A., Jacquet, H.:
Automorphic forms and automorphic representations.\
In: Automorphic forms, representations and L-functions, 
Proc.\ Symp.\ Pure Math. 
Am.\ Math.\ Soc., Corvallis/Oregon 1977, 
Proc.\ Symp.\ Pure Math.\ 33, 1, pp.\ 189--202 (1979)

\bibitem[Br-Ti]{BT}Bruhat, F., Tits, J.:  
Groupes reductifs sur un corps local.\ 
Publ.\ Math., Inst.\ Hautes \'Etud.\ Sci.\ 41, 5--251 (1972)

\bibitem[De-Hu]{De-Hu}Deligne, P., Husemoller, D.:
Survey of {D}rinfel'd modules.\
In:
Current trends in arithmetical algebraic geometry,
Proc.\ Summer Res.\ Conf.,
Arcata/Calif.\ 1985,
Contemp.\ Math.\ 67, pp.\ 25--91 (1987)

\bibitem[Gr]{Gra}Grayson,  D.\ H.:
Finite generation of $K$-groups of a curve over a 
finite field (after Daniel Quillen).
In:
Algebraic $K$-theory, Part I (Oberwolfach, 1980), 
pp.\ 69--90,
Lect.\ Notes Math., 966,
Springer, Berlin-New York
(1982)

\bibitem[Ha]{Haefliger}Haefliger, A.:
Introduction to piecewise linear intersection homology. 
In:
Intersection cohomology (Bern, 1983), pp.\ 1--21,
Progr.\ Math.\ 50, Birkh\"auser Boston, Boston 
(1984) 


\bibitem[Har]{Harder}Harder, G.:
Die Kohomologie S-arithmetischer Gruppen \"uber Funktionenk\"orpern. (German)
Invent.\ Math.\ 42, 135--175 (1977)


\bibitem[Har-Na]{HN}Harder,  G., Narasimhan, M.~S.: 
On the cohomology groups of moduli spaces of 
vector bundles on curves.
Math.\ Ann.\ 212, 215--248
 (1974/75)


\bibitem[Hatc]{Hatcher}
Hatcher, A.:
Algebraic topology.
Cambridge University Press,
Cambridge
(2002)

\bibitem[Hatt]{Hattori}
Hattori, A.:
Topology, I-III. 3rd ed. (Japanese)
Iwanami Shoten Kiso S\=ugaku, 19.\ Kikagaku, ii.\ 
Iwanami Shoten, Tokyo 
(1987)


\bibitem[Iv]{Iversen}
Iversen, B.: 
Cohomology of sheaves.\ 
Lect.\ Notes Ser., 
Aarhus Univ.\ 55 
(1984)

\bibitem[KeWe]{KeWe}
Kegel, O.\ H., Wehrfritz, B.\ A.\ F.:
Locally finite groups.\
North-Holland Mathematical Library.\ Vol.\ 3.\ 
Amsterdam-London: North- Holland Publishing Comp.; New York: 
American Elsevier Publishing Comp., Inc.\ (1973)
 
\bibitem[Ko-Ya]{KY:Zeta elements}
Kondo, S., Yasuda, S.:
Zeta elements in the $K$-theory of Drinfeld modular varieties. 
Math.\ Ann.\ 354, no.\ 2, 529--587 (2012)

\bibitem[La1]{Laumon1}
Laumon, G.: 
Cohomology of Drinfeld modular varieties.\ 
Part 1: Geometry, counting of points and local harmonic analysis.\ 
Cambridge Studies in Advanced Mathematics.\ 41.\ 
Cambridge University Press,
Cambridge 
(1996)

\bibitem[La2]{Laumon2}
Laumon, G.: 
Cohomology of Drinfeld modular varieties.\
Part II: 
Automorphic forms, trace formulas and Langlands correspondence.\ 
With an appendix by Jean-Loup Waldspurger.\
Cambridge Studies in Advanced Mathematics.\ 56.\ 
Cambridge University Press,
Cambridge 
(1997)

\bibitem[Mo-Wa]{MW} Moeglin, C., Waldspurger, J.-L.:
Le spectre r\'esiduel de $GL(n)$.\
Ann.\ Scient.\ \'Ec.\ Norm.\ Sup.\ $4^e$ s\'erie.\
22, 605--674
(1989)

\bibitem[We1]{We2}Werner, A.:
Compactification of the {B}ruhat-{T}its building of {PGL} by lattices of smaller rank.\
Doc.\ Math., J.\ DMV 6, 315--341 (2001)

\bibitem[We2]{We1}Werner, A.:
Compactification of the {B}ruhat-{T}its building of {PGL} by seminorms.\
Math.\ Z. 248, no.\ 3, 511--526 (2004)


\bibitem[Ze]{Zelevinsky}
Zelevinsky, A.\ V.: 
Induced representations of reductive $\mathfrak{p}$-adic
groups.\ II.\ On irreducible representations of $\GL(n)$.\
Ann.\ Sci.\ E.\ N.\ S.\ 4e s\'erie. 13, no.\ 2, 
165--210 (1980)

\end{thebibliography}
\end{document}